\documentclass[AMA,Times1COL]{WileyNJDv5} %STIX1COL,STIX2COL,STIXSMALL
\usepackage[scr = boondoxupr]{mathalpha}
\usepackage{tikz}
\usepackage{pgfplots}
\usepackage{pgfplotstable}
\pgfplotsset{compat=newest}
\usetikzlibrary{shapes, arrows, chains, calc, positioning, quotes, fit, arrows.meta, decorations.pathreplacing,pgfplots.groupplots}

\usepackage{graphicx}
\usepackage{caption}
\usepackage{subcaption}
\usepackage{float}
\usepackage{mathtools}
\usepackage{bbm}
\usepackage{algorithm}
\usepackage{algorithmicx}
\usepackage{algpseudocode}

\algnewcommand{\Inputs}[1]{%
  \State \textbf{inputs:}
  \Statex \qquad \hspace*{\algorithmicindent}{\raggedright #1}
}
\algnewcommand{\Initialize}[1]{%
  \State \textbf{initialize:}
  \Statex \hspace*{\algorithmicindent}{\raggedright #1}
}
\algnewcommand{\LineComment}[2]{\Statex #1 \(\triangleright\) #2}

\DeclareMathOperator*{\argmin}{arg\,min}

\newcommand{\diag}{\operatorname{diag}} 
\usepackage{cleveref}
\crefname{section}{Section}{Sections}
\crefname{equation}{}{}
\crefname{algorithm}{Algorithm}{Algorithms}
\crefname{corollary}{Corollary}{Corollaries}
\crefname{figure}{Fig}{Figs}
\crefname{assumption}{Assumption}{Assumptions}
\crefname{table}{Table}{Tables}
\crefname{theorem}{Theorem}{Theorems}
\crefname{lemma}{Lemma}{Lemmas}
\newlength{\mywidth}
\tikzstyle{block} = [draw, rectangle,
minimum height=0.1\mywidth, minimum width=0.1\mywidth]
\tikzstyle{sum} = [draw, circle, node distance=1\mywidth]
\tikzstyle{input} = [coordinate]
\tikzstyle{output} = [coordinate]
\tikzstyle{pinstyle} = [pin edge={to-,thin,black}]

\newcounter{inlineenum}
\renewcommand{\theinlineenum}{\roman{inlineenum}}
  \newenvironment{inlineenum}
    {\unskip\ignorespaces\setcounter{inlineenum}{0}%
     \renewcommand{\item}{\refstepcounter{inlineenum}{\textit{\theinlineenum})~}}}
    {\ignorespacesafterend}

\articletype{Article Type}%

\received{Date Month Year}
\revised{Date Month Year}
\accepted{Date Month Year}
\journal{Journal}
\volume{00}
\copyyear{2023}
\startpage{1}

\begin{document}
% the title name
\def\mytitle{Distributed Optimization of Finite Condition Number for Laplacian Matrix in Multi-Agent Systems}

\title{\mytitle}

\author[1]{Yicheng Xu}

\author[1]{Faryar Jabbari}

\authormark{Xu \textsc{et al.}}
\titlemark{\mytitle}

\address[1]{\orgdiv{Department of Mechanical and Aerospace Engineering}, \orgname{University of California Irvine}, \orgaddress{\street{Engineering Gateway 4200}, \city{Irvine}, \state{CA}, \country{USA}}}

% \address[2]{\orgdiv{Department Name}, \orgname{Institution Name}, \orgaddress{\state{State Name}, \country{Country Name}}}

% \address[3]{\orgdiv{Department Name}, \orgname{Institution Name}, \orgaddress{\state{State Name}, \country{Country Name}}}

\corres{Yicheng Xu\email{yichex7@uci.edu}}

% \presentaddress{Engineering Gateway 4200, Irvine, CA 92697}

%\fundingInfo{Text}
%\JELinfo{ejlje}

\abstract[Abstract]{
    This paper addresses the distributed optimization of the finite condition number of the Laplacian matrix in multi-agent systems. The finite condition number, defined as the ratio of the largest to the second smallest eigenvalue of the Laplacian matrix, plays an important role in determining the convergence rate and performance of consensus algorithms, especially in discrete-time implementations. We propose a fully distributed algorithm by regulating the node weights. The approach leverages max consensus, distributed power iteration, and consensus-based normalization for eigenvalue and eigenvector estimation, requiring only local communication and computation. Simulation results demonstrate that the proposed method achieves performance comparable to centralized LMI-based optimization, significantly improving consensus speed and multi-agent system performance. The framework can be extended to edge weight optimization and the scenarios with non-simple eigenvalues, highlighting its scalability and practical applicability for large-scale networked systems.
}


\jnlcitation{\cname{%
        \author{Y Xu.}, and
        \author{F Jabbari}}.
    \ctitle{\mytitle} \cjournal{\it J Comput Phys.} \cvol{2025;00(00):1--18}.}

\maketitle

\renewcommand\thefootnote{}
% \footnotetext{\textbf{Abbreviations:} ANA, anti-nuclear antibodies; APC, antigen-presenting cells; IRF, interferon regulatory factor.}

\renewcommand\thefootnote{\fnsymbol{footnote}}
\setcounter{footnote}{1}

\section{Introduction}
\label{sec: Introduction}
The study of Multi-Agent Systems (MAS) has drawn a significant attention in recent years due to their wide-ranging applications in areas such as distributed control, sensor networks and robotics \cite{Zhang2024a,zhang2024,Gao2024}. A fundamental aspect of MAS is the underlying communication topology, typically modeled as a graph, which represents how information is exchanged among agents \cite{Mesbahi2010}. The Laplacian matrix associated with this graph plays a pivotal role in determining the collective behavior of the network \cite{Liu2019,Su2022}.

A key challenge in MAS is to understand how the spectral properties of the Laplacian matrix impacts the feasibility of consensus \cite{Griparic2022,Hengster-Movric2015}. In continuous-time control, consensus can generally be achieved for any connected graph under standard controllability assumptions \cite{HongweiZhang2011,Li2010,Giordano2016}. However, in real-life applications, actuators and controllers are implemented in discrete time due to digital hardware constraints. For discrete-time controllers, even within the state feedback framework, not every graph structure guarantees consensus for the agents with open-loop unstable modes. It turns out that the finite condition number \cite{GuoxiangGu2012,You2011,Feng2020}, i.e., the largest eigenvalue over the second smallest (Fiedler) eigenvalue, of the Laplacian matrix is a critical factor in determining the feasibility of consensus.

Furthermore, the finite condition number of the Laplacian matrix also plays a crucial role in improving the overall MAS performance. The convergence rate of the average consensus algorithm can be directly expressed in terms of the finite condition number \cite{Xiao2004}. A small finite condition number implies the faster convergence rate. In the context of output feedback synthesis, for example, the finite condition number influences the feasibility region of the Linear Matrix Inequalities (LMIs) that characterize the $l_2$ gain of the closed-loop system performance \cite{Xu2024a}. Optimizing the finite condition number of the Laplacian matrix can be significantly beneficial for enhancing the performance of MAS (\cref{sec: Motivation and Problem Statement}). There are other applications of the finite condition number in MAS, such as improving performance in event-triggered control \cite{Liang2023} and synchronization in complex networks \cite{SORRENTINO2007,Barahona2002}.

While some prior works have explored the modification of network topology, such as the addition or removal of edges to enhance MAS performance \cite{Kim2014,Cao2025}, there have been researches over the past decades on optimizing the edge weights of the Laplacian matrix, particularly with the objective of maximizing the Fiedler eigenvalue \cite{Tam2020,Ogiwara2017,Nagarajan2015,Somisetty2025,Tavasoli2024}. Optimizing the finite condition number of the Laplacian matrix, however, has received relatively limited attention. Among the most closely related works, Shafi et al. \cite{Shafi2012} addressed the problem by reformulating it centrally into LMI. The resulting search is not fully convex. which leads to an iterative algorithm that requires continuously solving the problem using a central Semi-Definite Programming (SDP) solver. Therefore, it is not very practical for large, distributed networks. Kempton et al. \cite{Kempton2017} introduced a distributed approach under edge weights based on multi-layer Ordinary Differential Equations (ODEs). While this method achieves a fully distributed implementation, its convergence is only locally guaranteed due to the inherent limitations of the primal-dual framework. Moreover, the use of fast-slow ODEs can present significant implementation challenges.

This work seeks to address the minimization of the finite condition number in a fully distributed manner in a discrete implementation, as the finite condition number directly impacts consensus and performance mostly in discrete-time. It can also be readily modified to maximize the Fiedler eigenvalue alone by simply changing the objective function in the optimization formulation. The discrete-time framework allows the gradient descent updates to be paused until the distributed eigenvector estimations converge, ensuring accurate optimization at each step. We put our focus on node-weighted Laplacian, instead of the edge-weighted ones, based on our experience, that the node-weighted Laplacian often achieves a lower finite condition number compared to edge-weighted approaches. The proposed method leverages augmented Lagrangian \cite{Birgin2014,Dimitri2015} to achieve global convergence. It is demonstrated that only local information and minimal global knowledge on the graph diameter is needed due to the use of max-consensus \cite{Lucchese2015,Nejad2009}. When the topology changes, the diameter of the graph can be distributed evenly by all agents \cite{Garin2012} without requiring a central authority. The algorithm incorporates distributed eigenvalue and eigenvector estimation using power iteration \cite{Corso1997,Lin2010} and consensus-based normalization.

Simulation results confirm that the distributed algorithm achieves performance nearly identical to centralized LMI-based optimization, yielding reductions in the finite condition number and marked improvements in consensus convergence rates and disturbance rejection capabilities \cite{Xu2024a}. The framework is robust to changes in network topology and can be extended to both node and edge weight optimization, as well as to scenarios involving non-simple eigenvalues \cite{Clarke1990}.  These results highlight the practical applicability and scalability of distributed optimization techniques for enhancing the performance and resilience of large-scale networked systems.

The remainder of this paper is structured as follows. In \cref{sec: Preliminary}, we present motivating examples and introduce the problem formulation. \cref{sec: Node weights} develops the main distributed optimization framework for node-weighted Laplacian matrices, including algorithmic details, distributed implementation, and simulation results. Extensions to edge weight optimization, technical considerations such as non-simple eigenvalues, robustness to network changes, and potential improvements are discussed in \cref{sec: Discussion and improvement}. The main contributions are summarized in \cref{sec: conclusions}. Supporting mathematical details and proofs are provided in the Appendix.

\section{Preliminary}
\label{sec: Preliminary}
\textbf{Notation:} $\|\cdot\|$ denotes the Euclidean (2-)norm, and $\|\cdot\|_\infty$ denotes the infinity norm for vectors. The ceiling function $\lceil x\rceil$ returns the smallest integer greater than or equal to $x$. $\operatorname{diag}(\cdot)$ denotes a diagonal matrix with the given vector on its diagonal. $\mathbb{R}$ denotes the set of real numbers, and $\mathbf{1}_N$ denotes the $N$-dimensional vector of all ones. The range space of a matrix $A$ is denoted by ${Range}(A)$, and the null space by ${Null}(A)$. $[\cdot]_i$ denotes the $i$th entry of a vector.

\subsection{Graph Theory}

A directed graph $\mathcal{G}$ is defined as an ordered pair $\left(\mathcal{V}, \mathcal{E}\right)$, where $\mathcal{V}$ represents the set of vertices, and $\mathcal{E} \subseteq \mathcal{V} \times \mathcal{V}$ denotes the set of edges. An edge from vertex $v_i$ to vertex $v_j$ is indicated by $(v_i, v_j) \in \mathcal{E}$. The neighborhood of a vertex $v_i$, denoted as $\mathcal{N}_i$, is the set of all vertices that have an edge directed towards $v_i$: $\mathcal{N}_i = \{v_j \in \mathcal{V} : (v_j, v_i) \in \mathcal{E}\}$. The number of the vertices is denoted by $N = \left|\mathcal{V}\right|$.  A graph is called undirected if, for every $(v_i, v_j) \in \mathcal{E}$, it also holds that $(v_j, v_i) \in \mathcal{E}$. The number of edges in an undirected graph is therefore even. The path between two vertices $v_i$ and $v_j$ is a sequence of vertices $v_{i_1}, v_{i_2}, \ldots, v_{i_k}$ such that $v_{i_1} = v_i$, $v_{i_k} = v_j$, and $(v_{i_{\ell}}, v_{i_{\ell+1}}) \in \mathcal{E}$ for all $\ell = 1, \ldots, k-1$. An undirected graph is called connected if there exists a path between any two vertices. The distance between two vertices $v_i$ and $v_j$ is defined as the fewest number of edges needed for a path connecting them, denoted as $\operatorname{dist}(v_i, v_j)$. The graph diameter of an undirected graph, denoted by $d(\mathcal{G})$, is defined as the maximum distance between any two vertices, i.e., $\displaystyle d(\mathcal{G}) = \max_{v_1, v_2 \in \mathcal{V}} \operatorname{dist}(v_1, v_2)$. The (in) adjacency matrix $\mathcal{A} = [a_{ij}] \in \mathbb{R}^{N \times N}$ of a graph $\mathcal{G}$ is a (0,1)-matrix, where $a_{ij} = 1$ if $(v_i, v_j) \in \mathcal{E}$, and $a_{ij} = 0$ otherwise. The (in) degree matrix, denoted by $\mathcal{D}$, is a diagonal matrix $\mathcal{D} = \operatorname{diag}(d_1, \ldots, d_N)$, where $d_i = \sum_{j=1}^{N} a_{ij}$. The (in) Laplacian matrix $\mathcal{L} \in \mathbb{R}^{N \times N}$ associated with graph $\mathcal{G}$, is defined as $\mathcal{L} = \mathcal{D} - \mathcal{A}$.
\subsection{Motivation and Problem Statement}
\label{sec: Motivation and Problem Statement}
We first present two motivating examples to show the importance of the finite condition number. Throughout the paper, the following assumption is imposed to restrict the eigenvalues of the Laplacian matrix to be real.
\begin{assumption}
    \label{assumption: undirected connected graph}
    The graph $\mathcal{G}$ is undirected and connected.
\end{assumption}

\subsubsection{Motivating Example: Average Consensus in MAS}
\label{subsubsec: Average Consensus in MAS}
In MAS, the diffusion dynamics or average consensus \cite{Tran2021,Jadbabaie2003,Gomez2013} of the agents can be modeled as
\begin{gather}
    \label{eq: discrete-time average consensus}
    x(k+1)-x(k)=-\frac{1}{r}\mathcal{L} x(k),
\end{gather}
where $r\in\mathbb{R}_+$ is the stepsize. The goal is to pick a proper $r\geq \underline{r}$, with $\underline{r}$ being the lower bound, such that the following consensus condition holds:
$$x(k)\to \frac{1}{N}\mathbf{1}_N\mathbf{1}_N^Tx(0).$$
Various choices of $r$ have been studied \cite{Zhang2011,Horn2012,Merris1998}, e.g., $r=N$ or $r=\displaystyle\max_{i=1,\cdots,N}\begin{vmatrix}
        d_i
    \end{vmatrix}$, where $d_i$ is the degree of node $i$. We are interested in finding the optimal $r$ that leads to the fastest convergence rate of \cref{eq: discrete-time average consensus}:
\begin{gather*}
    \rho^*=\min_{r\in\mathbb{R}_+}\rho\left(
    I-\frac{1}{r}\mathcal{L}-\frac{1}{N}\mathbf{1}_N\mathbf{1}_N^T
    \right),
\end{gather*}
which is equivalent to the following problem:
\begin{gather}
    \label{eq: min max problem for average consensus}
    \min_{r\in\mathbb{R}_+} \max_{i\in \left\{
        2,\cdots, N
        \right\}}\begin{vmatrix}
        1-\frac{1}{r} \lambda_i(\mathcal{L})
    \end{vmatrix},
\end{gather}
where $\lambda_i(\mathcal{L})$ is the $i$th eigenvalue of $\mathcal{L}$ satisfying $0=\begin{vmatrix}
        \lambda_1(\mathcal{L})
    \end{vmatrix}< \begin{vmatrix}
        \lambda_2(\mathcal{L})
    \end{vmatrix}\leq \cdots \leq \begin{vmatrix}
        \lambda_N(\mathcal{L})
    \end{vmatrix}$. With \cref{assumption: undirected connected graph}, the optimal stepsize $r$ can be obtained (The proof is given in \cref{sec: Optimization Approaches.1a}):
\begin{gather}
    \label{eq: optimal stepsize for real average consensus}
    r^*=\frac{\lambda_N(\mathcal{L})+\lambda_2(\mathcal{L})}{2},\ \rho^*=\frac{\lambda_N(\mathcal{L})-\lambda_2(\mathcal{L})}{\lambda_N(\mathcal{L})+\lambda_2(\mathcal{L})}=1-\frac{2}{1+\frac{\lambda_N(\mathcal{L})}{\lambda_2(\mathcal{L})}}.
\end{gather}

The larger the finite condition number is, the slower the convergence rate becomes.
\subsubsection{Dynamic Output Feedback Synthesis in MAS}
\label{subsubsec: Dynamic Output Feedback Synthesis in MAS}
The synthesis of dynamic output feedback controllers is an important problem in multi-agent systems (MAS), as it provides a general framework for achieving desired closed-loop performance. Consider a network of homogeneous agents, each governed by the following dynamics:
\begingroup
\arraycolsep=0pt
\begin{equation}
    \label{agent dynamic}
    \begin{array}{rcrcr}
        x_i(k+1) = & A   & x_i(k) & +B_1  \xi_i(k)  & +B_2u_i(k), \\
        y_i(k)   = & C_2 & x_i(k) & +D    \xi_i(k),
    \end{array}
\end{equation}
\endgroup
where $x_i$ is the states, $y_i$ is the output, $\xi_i$ is the exogenous disturbance, and $u_i$ is the control input of agent $i$. Denote the relative error as $\varepsilon_i=x_i-\frac{1}{N}\sum_{j=1}^N x_j$, and the measurement error is $z_i=C_1\varepsilon_i$. A typical disturbance $\xi$ to general output measurement $z$ attenuation problem uses $l_2$ norm, and aims to design a distributed controller for all $i\in\{1,\cdots, N\}$:
\begin{gather}
    \label{eq: Generalized modified DOF with network difference signal}
    \begin{aligned}
        {x}_{ci}(k+1)= & A_c{x}_{ci}(k) +B_c\sum_{j\in\mathcal{N}_i}y_i(k)-y_j(k), \\
        u_i(k)=        & C_c{x}_{ci}(k)+D_c\sum_{j\in\mathcal{N}_i}y_i(k)-y_j(k),
    \end{aligned}
\end{gather}
such that \begin{inlineenum}
    \item $\lim_{k\to\infty} \varepsilon_i(k) = 0$ with zero $\xi(0)$,
    \item and $\|z(k)\|_2\leq \gamma \|\xi(k)\|_2$ with zero initial condition.
\end{inlineenum}
As shown in \cite{Xu2024a}, this entials to solving the following LMIs, in which $\lambda_i$ is taken at two points $\lambda_i=\lambda_2(\mathcal{L})$ and $\lambda_N(\mathcal{L})$:
\begin{gather}
    \label{CL H infty optimization}
    \begin{bmatrix}
        Y                     & *         & *                    & *   & * & *        \\
        I_n                   & X         & *                    & *   & * & *        \\
        0                     & 0         & \gamma I             & *   & * & *        \\
        YA+\lambda_iW_oC_2    & L         & YB_1+\lambda_iW_oD   & Y   & * & *        \\
        A+\lambda_i B_2D_cC_2 & AX+B_2W_c & B_1+\lambda_iB_2D_cD & I_n & X & *        \\
        C_1                   & C_1X      & 0                    & 0   & 0 & \gamma I
    \end{bmatrix}
    \succ0,
\end{gather}
where $X,Y,W_o,W_c,D_c,L,\gamma$ are the decision variables. Through a simple scaling of the decision variables:
$$W_o,D_c \to \frac{1}{\lambda_2(\mathcal{L})}W_o,\frac{1}{\lambda_2(\mathcal{L})}D_c, $$
the $l_2$ gain is smaller than $\gamma$ if \cref{CL H infty optimization} can be solved at $\lambda_i=1$ and $\frac{\lambda_N(\mathcal{L})}{\lambda_2(\mathcal{L})}$.
When $\lambda_2(\mathcal{L})=\lambda_N(\mathcal{L})$, \cref{CL H infty optimization} becomes the key inequality associated with a single-agent problem, which is guaranteed to be feasible for some finite $\gamma$ under traditional controllability and observability assumptions \cite{Masubuchi1998,Scherer1996}. However, for a general graph, $\lambda_2(\mathcal{L})\neq \lambda_N(\mathcal{L})$, and the ratio between the two eigenvalues is a key factor in the synthesis problem. A lower finite condition number leads to a larger feasible region for the decision variables, which results in better attenuation performance of the closed-loop system. Therefore, the optimization of the weights on the edges or nodes will lead to a better performance of the MAS.

\subsubsection{Problem statement}
Often, for unweighted graphs, the node (or edge) weights are assumed to be one. One can relax such restriction by assigning (non-negative) weights to the nodes or edges to lower the finite condition number, which improves the system behavior. The weights are expected to be optimized to achieve:
\begin{gather}
    \label{eq: finite conditional number problem for weighted Laplacian}
    \min_{{\mathbf{w}}} \frac{\lambda_N({\mathcal{L}}_{{\mathbf{w}}})}{\lambda_2({\mathcal{L}}_{{\mathbf{w}}})},
\end{gather}
where $\mathbf{w}$ is the weights and $\mathcal{L}_\mathbf{w}$ is the weighted Laplacian matrix. Furthermore, to meet the practical requirement of the applications, the algorithm developed in this paper is expected to be fully distributed. The main goal of this paper is to develop an algorithm that
\begin{inlineenum}
    \item converges to the minimum finite condition number
    \item is fully distributed
    \item requires minimum information regarding graph configuration
\end{inlineenum}

\subsection{Optimizing Node Weights}
\label{subsubsec: Adjusting the Weights on the Nodes}
Tuning edge weights have been studied in the literature \cite{Tam2020,Ogiwara2017,Nagarajan2015,Somisetty2025,Tavasoli2024} for a variety of problems. An alternative and less explored approach is to minimize the finite condition number of the Laplacian matrix by adjusting node weights. This node-weighted approach offers a significant advantage as the number of nodes, $N$, is typically much smaller than the number of edges, $M$, resulting in fewer variables to optimize. To be specific, each node $i$ is assigned a single non-negative weight ${w}_i\geq0$, which is uniformly applied to all incoming information from neighboring agents. The associated node weighted Laplacian matrix is then
\begin{gather}
    \label{eq: node weighted Laplacian matrix}
    {\mathcal{L}}_{\mathbf{w}}=\diag(\mathbf{w})\mathcal{L},
\end{gather}
where $\mathbf{w}\in\mathbb{R}^N_+$. Note that $\mathbf{1}_N$ as the eigenvector corresponding to the zero eigenvalue, i.e., $\mathbf{1}_N\in Null\left(
    \mathcal{L}_{\mathbf{w}}
    \right)$. In practice, each node can simply multiply the information it receives from it neighbors by its assigned weight, before using the information,making node-weighting straightforward to implement.The introduction of $\diag({\mathbf{w}})$ destroys the symmetry of the Laplacian matrix. The following lemma demonstrates all eigenvalues $\lambda({\mathcal{L}}_{{\mathbf{w}}})$ are real.

\begin{lemma}
    \label{lemma: node weighted Laplacian matrix}
    Suppose the node weight $\mathbf{w}\in\mathbb{R}^N$ is entry-wise positive, then the eigenvalues of matrix ${\mathcal{L}}_{{\mathbf{w}}}$ are all real. Furthermore, suppose \cref{assumption: undirected connected graph} holds, then the second-smallest eigenvalue of ${\mathcal{L}}_{{\mathbf{w}}}$ is positive.
\end{lemma}

Consider the following matrix which is congruent to $\mathcal{L}$, thus it has no negative eigenvalue and one zero eigenvalue:
\begin{gather}
    \label{eq: node weighted Laplacian matrix 2}
    \hat{\mathcal{L}}_\mathbf{w}=\diag({\mathbf{w}})^{\frac{1}{2}}\mathcal{L}\diag({\mathbf{w}})^{\frac{1}{2}}.
\end{gather}
$\hat{\mathcal{L}}_\mathbf{w}$ is similar to $\mathcal{L}_{\mathbf{w}}$ with the change of the basis matrix $\diag({\mathbf{w}})^{\frac{1}{2}}$ establishing $\mathcal{L}_\mathbf{w}$ having all real positive eigenvalues, with the exception of a zero eigenvalue. Note the null space $Null\left(
    \hat{\mathcal{L}}_\mathbf{w}
    \right)=span\{
    \diag\{\mathbf{w}\}^{-\frac{1}{2}}\mathbf{1}_N
    \}$ and the range space of $\hat{\mathcal{L}}_\mathbf{w}$ is
\begin{gather}
    \label{eq: range space of the node weighted Laplacian matrix}
    \mathcal{S}:=Range\left(
    \hat{\mathcal{L}}_\mathbf{w}
    \right)=\left\{
    v\in \mathbb{R}^{N} \mid v^T \diag(\mathbf{w})^{-\frac{1}{2}}\mathbf{1}_N=0
    \right\}.
\end{gather}
\begin{remark}
    With \cref{lemma: node weighted Laplacian matrix}, \cref{eq: finite conditional number problem for weighted Laplacian} is well posed as the loss function becomes real.
    Since $\hat{\mathcal{L}}_{\mathbf{w}}$ exhibits better numerical properties than $\mathcal{L}_{\mathbf{w}}$, $\hat{\mathcal{L}}_{\mathbf{w}}$ is used to replace $\mathcal{L}_{\mathbf{w}}$ in \cref{eq: finite conditional number problem for weighted Laplacian},which as demonstrated later, does not affect the distributed nature of the algorithm.

    The finite condition number in \cref{eq: finite conditional number problem for weighted Laplacian} is, only quasi convex. For a given $\mathbf{w}\in\mathbb{R}^N_+$, scalar multiplication by any $\alpha>0$ preserves the ratio, as it scales both the Fiedler value and the largest eigenvalue. Consequently, there exist infinitely many optimal points for \cref{eq: finite conditional number problem for weighted Laplacian}, differing only by a scalar factor. Therefore, without loss of generality, an extra constraint is posed on the weights to force the uniqueness of the optimal solution and to ensure that the graph remains connected. For numerical stability, the Fiedler value is chosen to be lower bounded by $1$. If a different lower bound other from $1$ is chosen, the optimization problem will lead to a similar optimal solution which differs only by a scalar factor.

    \begin{gather}
        \label{eq: modified finite conditional number problem for node weighted Laplacian}
        \begin{aligned}
            \min_{\hat{\mathbf{w}}\in\mathbb{R}^N} & \quad \lambda_N(\hat{\mathcal{L}}_{{\mathbf{w}}}), \\
            \text{s.t.}                            & \quad
            \left\{
            \begin{array}{l}
                1-\lambda_2(\hat{\mathcal{L}}_{{\mathbf{w}}})\leq 0, \\
                {\mathbf{w}}\geq 0.
            \end{array}
            \right.
        \end{aligned}
    \end{gather}
\end{remark}

\begin{lemma}
    \label{lemma: convexity of the optimization problem}
    The optimization problem \cref{eq: modified finite conditional number problem for node weighted Laplacian} is a convex problem.
\end{lemma}

\begin{proof}
    Denote the Cholesky decomposition of $\mathcal{L}$ as $\mathcal{L}=CC^T$. Because of \cref{assumption: undirected connected graph}, $\mathcal{L}$ is of rank $N-1$, which results to $C\in\mathbb{R}^{N\times (N-1)}$. Since $\hat{\mathcal{L}}_{\mathbf{w}}\succeq 0$, it holds that:
    \begin{gather}
        \label{eq: equivalence of eigenvalue}
        \left.
        \begin{array}{r}
            \diag({\mathbf{w}})^{\frac{1}{2}}CC^T\diag({\mathbf{w}})^{\frac{1}{2}}\succeq 0 \\
            \sigma\left[
            C^T\diag({\mathbf{w}})^{\frac{1}{2}}
            \right]=\sigma\left[
            \diag({\mathbf{w}})^{\frac{1}{2}}C
            \right]
        \end{array}
        \right\}\Rightarrow \lambda_{i+1}\left[
        \diag({\mathbf{w}})^{\frac{1}{2}}CC^T\diag({\mathbf{w}})^{\frac{1}{2}}
        \right]=\lambda_{i}\left[
            C^T\diag({\mathbf{w}})C
            \right],\ \forall i=1,\cdots,N-1,
    \end{gather}
    with the notation that $\lambda_i$ is the $i$th smallest eigenvalue.
    \cref{eq: equivalence of eigenvalue} indicates that the non-zero eigenvalues of the matrix $\hat{\mathcal{L}}_{\mathbf{w}}$ are the same as the eigenvalues of the matrix $C\diag{\mathbf{w}}C^T\in\mathbb{R}^{(N-1)\times (N-1)}$, with the zero eigenvalue being dropped. Because of the concavity of the smallest eigenvalue function in a positive semidefinite matrix, $\lambda_{2}(\hat{\mathcal{L}}_{\mathbf{w}})=\lambda_{1}\left(
        C\diag{\mathbf{w}}C^T
        \right)$ is concave. Similarly, $\lambda_{N}(\hat{\mathcal{L}}_{\mathbf{w}})=\lambda_{N-1}\left(
        C\diag{\mathbf{w}}C^T
        \right)$ is convex. Therefore, \cref{eq: modified finite conditional number problem for node weighted Laplacian} is a convex optimization problem.
\end{proof}

The optimization problem in \cref{eq: modified finite conditional number problem for node weighted Laplacian} can be equivalently formulated centrally, as an LMI problem, as outlined in \cref{subsec: LMI formulation}. While LMI-based approaches are typically centralized and thus not the focus of this work, the LMI formulation is included for completeness and to provide a benchmark for evaluating the performance of the proposed distributed algorithm. For solving \cref{eq: modified finite conditional number problem for node weighted Laplacian} distributedly, one can apply the augmented Lagrangian method to obtain an iterative algorithm (\cref{subsec: Augmented Lagrangian Method}). Consider the following augmented Lagrangian:
\begin{gather}
    \label{eq: Lagrangian function for node weighted Laplacian}
    {\mathscr{L}}_{{\rho}}({\mathbf{w}},{\sigma})=
    \lambda_N(\hat{\mathcal{L}}_{{\mathbf{w}}})+\frac{{\rho}}{2}\left[
        {\max}\left\{
        0,1-\lambda_2\left(
        \hat{\mathcal{L}}_{{\mathbf{w}}}
        \right)+\frac{{\sigma}}{{\rho}}
        \right\}
        \right]^2,
\end{gather}
where ${\sigma}\in\mathbb{R}$ is the Lagrange multiplier and ${\rho}$ is the penalty parameter. From \cref{lemma: convexity of the optimization problem}, both $\lambda_N(\hat{\mathcal{L}}_\mathbf{w})$ and $-\lambda_2(\hat{\mathcal{L}}_\mathbf{w})$ are convex functions in $\mathbf{w}$. Taking point-wise maximum with $0$ and function $1-\lambda_2(\hat{\mathcal{L}}_\mathbf{w})+\frac{\sigma}{\rho}$ indicates that $\left[
        {\max}\left\{
        0,1-\lambda_2\left(
        \hat{\mathcal{L}}_{{\mathbf{w}}}
        \right)+\frac{{\sigma}}{{\rho}}
        \right\}
        \right]$ is convex in $\mathbf{w}$. The square function is convex and non-decreasing on $[0,\infty]$, therefore the composite function $\left[
        {\max}\left\{
        0,1-\lambda_2\left(
        \hat{\mathcal{L}}_{{\mathbf{w}}}
        \right)+\frac{{\sigma}}{{\rho}}
        \right\}
        \right]^2$ is also convex in $\mathbf{w}$, which establishes $\mathscr{L}_{{\rho}}({\mathbf{w}},{\sigma})$ is convex in $\mathbf{w}$ for any $\sigma\geq 0,\rho>0$.

The associated iterative algorithm for solving \cref{eq: modified finite conditional number problem for node weighted Laplacian} is:
\begin{subequations}
    \label{eq: overall iterative algorithm for node weighted Laplacian}
    \begin{align}
        \label{eq: overall iterative algorithm for node weighted Laplacian 1}
        {\mathbf{w}}^{k+1}= & \argmin_{{\mathbf{w}}\geq 0}{\mathscr{L}}_{{\rho}}({\mathbf{w}},{\sigma}^k), \\
        \label{eq: overall iterative algorithm for node weighted Laplacian 2}
        {\sigma}^{k+1}=     & \max\left\{
        {\sigma}^k+{\rho}\left[
            1-\lambda_2\left(
            \hat{\mathcal{L}}_{{\mathbf{w}}^{k+1}}
            \right)
            \right],0
        \right\}.
    \end{align}
\end{subequations}

Note that the main step \cref{eq: overall iterative algorithm for node weighted Laplacian 1} is an optimization problem that can be solved using the projected gradient descent method (See \cref{subsec: Projected Gradient Descent}).

\section{Distributed Node Weights optimization}
\label{sec: Node weights}
In this section, we consider the optimization of the node weights as stated in \cref{eq: modified finite conditional number problem for node weighted Laplacian}. A fully distributed optimization algorithm is developed to realize the iterative algorithm presented in \cref{eq: overall iterative algorithm for node weighted Laplacian}. First, an overview of the process is given in the form of a flowchart. Then, the detailed algorithm is presented in the subsequent subsection.
\subsection{Optimization Process Overview}
Algorithm \cref{eq: overall iterative algorithm for node weighted Laplacian} incorporates three optimization iterations. Step \cref{eq: overall iterative algorithm for node weighted Laplacian 1} solves a subproblem dependent on a fixed $\sigma^k$, from prior iteration, while step \cref{eq: overall iterative algorithm for node weighted Laplacian 2} updates $\sigma^k$ based on the result of the solved sub-optimal point $\mathbf{w}^{k+1}$. Step \cref{eq: overall iterative algorithm for node weighted Laplacian 2} is simple, as $\lambda_2(\hat{\mathcal{L}}_{\mathbf{w}^{k+1}})$ is obtained during \cref{eq: overall iterative algorithm for node weighted Laplacian 1} iteration. The key computational step is \cref{eq: overall iterative algorithm for node weighted Laplacian 1}, which is obtained via projected gradient descent, denoted here as `outer iteration'. Time step $t$ is used to represent the iteration number of this outer iteration. However, as shown below, to obtain updated values for $\mathbf{w}$ at each iteration, the gradient of the augmented Lagrangian function is required. The gradient is a function of the eigenvectors of the weighted Laplacian matrix, which depend on the weights. As a result, after each update of $\mathbf{w}$, between $t$ and $t+1$, new eigenpairs for $\lambda_2(\hat{\mathcal{L}}_{\mathbf{w}(t)})$ and $\lambda_N(\hat{\mathcal{L}}_{\mathbf{w}(t)})$ need to obtained based on the updated weights. This is referred to as the inner iteration. Time step $T$ is initialized to be $0$ after each time $t$ becomes $t+1$. The detailed layout is given in the following, in which all steps can be accomplished in a distributed fashion.

\begin{itemize}
    \begin{minipage}{.49\linewidth}
        \item Augmented Lagrangian updates for \cref{eq: Lagrangian function for node weighted Laplacian} using \cref{eq: overall iterative algorithm for node weighted Laplacian}
    \end{minipage}
    \begin{minipage}{.29\linewidth}
        $\Leftarrow$ Augmented Lagrangian Update
    \end{minipage}
    \begin{minipage}{.1\linewidth}
        \centering
        $k\gets k+1$
    \end{minipage}
    % \vspace{3mm}

    \begin{itemize}
        \begin{minipage}{.49\linewidth}
            \item[-] Gradient descent iteration for \cref{eq: distributed node weight update}
        \end{minipage}
        \begin{minipage}{.29\linewidth}
            $\Leftarrow$ Outer iteration
        \end{minipage}
        \begin{minipage}{.1\linewidth}
            \centering
            $t\gets t+1$
        \end{minipage}
        \vspace{2mm}

        \begin{minipage}{.49\linewidth}
            \begin{itemize}
                \item Power iteration for the largest and Fielder eigenpairs using \cref{eq: central distributed power iteration using infty norm,eq: central Fiedler eigen-pair estimator}
                \item Normalization of the eigenvectors using \cref{Normalization procedure}
            \end{itemize}
        \end{minipage}
        \begin{minipage}{.29\linewidth}
            $\Leftarrow$ Inner iteration
        \end{minipage}
        \begin{minipage}{.1\linewidth}
            \centering
            $T\gets T+1$
        \end{minipage}
    \end{itemize}
\end{itemize}

For better understanding, the overall process is also depicted in \cref{fig: Node weight Overall Process}. The inner iteration is shown in the dotted box, the outer iteration in the solid box, and the Lagrangian update within the dashed-dot box. The feedback loop updates the largest eigenvalue of the weighted Laplacian matrix, while the update of $\sigma$ outside the outer iteration is used to initialize the next sub-optimal problem.

\begin{figure}[htbp]
    \vspace{-2mm}
    \scriptsize
    \centering
    \def\arrowOffSet{0.25\mywidth}
    \def\ratioOffSet{0.3}
    \begin{tikzpicture}[auto, node distance=1\mywidth,>=latex',scale=0.8]
        %% largest eigen-pair estimator
        \node [block, align=center](Largest){Largest Eigen-pair\\Estimator};
        %% initailization
        \node [block, align=center, left of = Largest, node distance = 3\mywidth](initialization){Set up ${\mathbf{w}}^0$\\Estimate $d({\mathcal{G}})$};
        %% smallest eigen-pair estimator
        \node [block, below of = Largest, align=center](Smallest){Fiedler Eigen-pair\\Estimator};
        %% normalization
        \node [block, right of = Largest, align=center,yshift=-0.5\mywidth,node distance=3\mywidth](Normalization){\\\\Distributed\\\\Normalization\\\\};
        %% weight update
        \node [block, below of = Smallest, node distance=1.85\mywidth, xshift = 1.5 \mywidth,align = center](weight update){\\\\Gradient Descent \\\\for \cref{eq: overall iterative algorithm for node weighted Laplacian 1}\\};
        %% update of sigma
        \node [block, below of = weight update, node distance= 2\mywidth](sigma update){one step update of \cref{eq: overall iterative algorithm for node weighted Laplacian 2}};
        %% feedback
        \node [input, left of = weight update, yshift=-\arrowOffSet, node distance =3\mywidth](feedback){};
        %% rho
        \path  let \p1 = (weight update.south), \p2 = (sigma update.north), \p3 = (Normalization) in node at ($(\x3,\ratioOffSet*\y1-\ratioOffSet*\y2+\y2)$) [block, align=center] (rho) {$\rho$};

        %% arrows
        \draw [->] ($(initialization.east)+(0,0.1\mywidth)$) -- ($(Largest.west)+(0,0.1\mywidth)$);
        \draw [dashed,->] (Largest) -- node{${\lambda_N}(\hat{\mathcal{L}}_{{\mathbf{w}}})$} (Smallest);
        \draw [->] (Smallest) |-node[right,pos=0.2]{${\lambda_2}(\hat{\mathcal{L}}_{{\mathbf{w}}})$} ($(weight update.west)+(0,\arrowOffSet)$);
        \draw [->] (Normalization) |- node[right,pos=0.2]{$\bar{v},\underline{v}$} node[left,pos=0.2,align=right]{$\|\bar{v}\|=1$\\[1mm]$\|\underline{v}\|=1$} ($(weight update.east)+(0,\arrowOffSet)$);
        % estimator to normalization
        \draw [->] let \p1 = (Largest.east),\p2=(Normalization.west) in ($(\x1,\y1)$) -- node[above]{$\bar{v}$} node [below] {$\|\bar{v}\|_\infty=1$} ($(\x2,\y1)$);
        \draw [->] let \p1 = (Smallest.east),\p2=(Normalization.west) in ($(\x1,\y1)$) -- node[above]{$\underline{v}$} node [below] {$\|\underline{v}\|_\infty=1$} ($(\x2,\y1)$);
        % feedback for inner iteration
        \draw [-] let \p1 = (weight update.west), \p2 =(feedback) in ($(\x1,\y2)$)-- (feedback);
        \draw [->] (feedback) |- ($(Largest.west)-(0,0.1\mywidth)$);
        % rho to update
        \draw [->] (rho) |- ($(weight update.east)+(0,-\arrowOffSet)$);
        \draw [->] (rho) |- (sigma update);
        % outer iteration
        \draw [->] ($(weight update.south)+(-\arrowOffSet,0)$) -- node [left, align = center, pos = 1-\ratioOffSet] {${\mathbf{w}}^{k+1}$\\when \cref{eq: overall iterative algorithm for node weighted Laplacian 1} meets \\ stopping criteria}($(sigma update.north)+(-\arrowOffSet,0)$);
        \draw [->] ($(sigma update.north)+(\arrowOffSet,0)$) -- node [right, align = center, pos = \ratioOffSet] {$\hat{\sigma}^{k+1}$} ($(weight update.south)+(\arrowOffSet,0)$);

        %% fit
        % \node[fit=(initialization) (Normalization) (weight update) (rho) (feedback) (sigma update), thick, draw, inner xsep=0.5\mywidth, inner ysep=0.5\mywidth](fit) {};

        \node[fit=(Largest) (Normalization) (weight update) (feedback), thick, draw, inner xsep=0.4\mywidth, inner ysep=0.4\mywidth](inner fit) {};
        \node[fit=(Largest) (Normalization) (Smallest), dotted, draw, inner xsep=0.2\mywidth, inner ysep=0.2\mywidth](fit) {};
        \node [fit= (weight update) (sigma update), loosely dash dot, draw, inner xsep=\mywidth, inner ysep=0.2\mywidth](Lagrangian fit) {};
    \end{tikzpicture}
    \caption{Overall process for Node Weight Optimization}
    \label{fig: Node weight Overall Process}
    % \caption{Structure of modified controller for amplitude and rate saturation }
    \vspace{-2mm}
\end{figure}
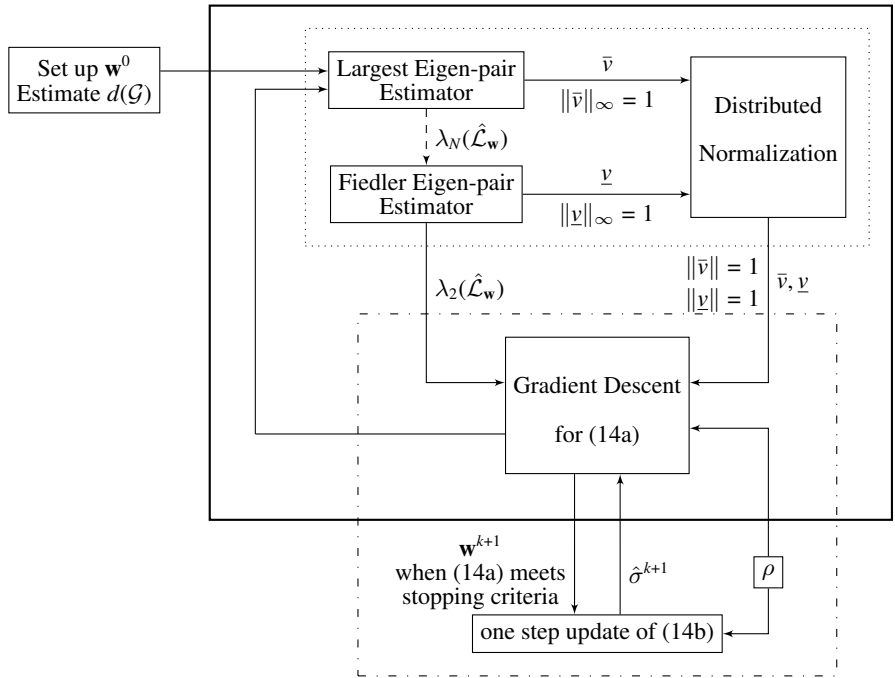

\begin{remark}
    The only prerequisite for initiating the distributed optimization algorithm is knowledge of the graph diameter $d(\mathcal{G})$. Importantly, $d(\mathcal{G})$  can be efficiently estimated in a distributed manner using standard algorithms such as max-consensus \cite{Garin2012} in $2d(\mathcal{G})$ steps. This opens the possibility to address occasional changes in the graph, e.g., link failures or node failures.
\end{remark}
\subsection{Distributed Node Weight Update}
\label{subsec: distributed node weight update}
Solving \cref{eq: overall iterative algorithm for node weighted Laplacian 1} via the gradient algorithm requires taking derivatives of the Lagrangian with respect to the $\mathbf{w}$. This needs the derivative of the largest and Fiedler eigenvalues. The following theorem guarantees that the projected gradient descent converges to the global minimizer of the sub-problem for \cref{eq: overall iterative algorithm for node weighted Laplacian 1}.
\begin{theorem}
    \label{theorem: distributed node weight update}
    Suppose ${\mathbf{w}}(0)> 0,{\sigma}\geq 0,{\rho}>0$. Consider the following dynamics:
    \begin{gather}
        \label{eq: distributed node weight update}
        {w}_i(t+1)=\max\left\{
        w_i(t)-\gamma \frac{\partial {\mathscr{L}}_{{\rho}}}{\partial {w}_i}\left(
        \mathbf{w}(t),\sigma
        \right),0
        \right\}
        ,\ i=1,\cdots,N,
    \end{gather}
    where $\gamma>0$ is the step size. For sufficiently small $\gamma$, the algorithm converges to the global optimal solution of \cref{eq: overall iterative algorithm for node weighted Laplacian 1}.
\end{theorem}

The convergence of \cref{eq: distributed node weight update} follows from that the augmented Lagrangian function is a convex function in $\mathbf{w}$ (demonstrated in \cref{subsubsec: Adjusting the Weights on the Nodes}) and the projected gradient descent converges to the global minimum of a convex function. To obtain the derivatives, we rely on the following lemma to find the derivative of the eigenvalue.
\begin{lemma}
    \label[lemma]{lemma: derivative of the eigenvalue} \cite{Magnus1985} Consider a Hermitian matrix $H(x)\in\mathbb{S}^N$ whose components are differentiable functions of $x$. Suppose the eigenvalue $\lambda(H(x))$ is simple, the differential of the eigenvalue with respect to $x$ is:
    \begin{gather*}
        \frac{\partial {\lambda}\left(H(x)\right)}{\partial x}={u}^T\frac{\partial H(x)}{\partial x}{u},
    \end{gather*}
    where ${u}$ is the eigenvectors corresponding to the eigenvalue $\lambda$ with $\|u\|=1$.
\end{lemma}

\begin{remark}
    In \cref{lemma: derivative of the eigenvalue}, the eigenvalue $\lambda(H(x))$ is required to be simple, for which the gradient is well defined. In some cases, the eigenvalues of interest may be repeated, e.g., star graph, cycle graph, complete graph, etc. The eigenvalues are continuous, but not differentiable \cite{Clarke1990} in such cases. We obtain the same result by leveraging the subgradient method. The details are discussed in \cref{sec: non-simple eigenvalues}.
\end{remark}

Denote the unit eigevector corresponding to a simple eigenvalue $\lambda(\hat{\mathcal{L}}_{{\mathbf{w}}})$ of weighted Laplacian matrix $\hat{\mathcal{L}}_{{\mathbf{w}}}$ as $v$.
With \cref{lemma: derivative of the eigenvalue,eq: node weighted Laplacian matrix 2}, the differential of the eigenvalue $\lambda(
    {\hat{\mathcal{L}}}_{{\mathbf{w}}}
    )$ with respect to the $i$th node weight ${w}_i$ is:
\begin{align*}
    \frac{\partial \lambda\left(
        {\hat{L}}_{{\mathbf{w}}}
        \right)}{\partial {w}_i}
     & ={v}^T\left\{
    \frac{\partial}{\partial w_i}\diag({\mathbf{w}})^{\frac{1}{2}}\mathcal{L}\diag({\mathbf{w}})^{\frac{1}{2}}+\diag({\mathbf{w}})^{\frac{1}{2}}\mathcal{L}\frac{\partial}{\partial w_i}\diag({\mathbf{w}})^{\frac{1}{2}}
    \right\}{v},                                                                                    \\
     & =\begin{bmatrix}
            0                       &
            \cdots                  &
            {w}_i^{-\frac{1}{2}}v_i &
            \cdots                  &
            0
        \end{bmatrix}\mathcal{L}\diag(\mathbf{w})^{\frac{1}{2}}v =v_i\sum_{j\in\mathcal{N}_i}\left[
    v_i-\left(\frac{w_j}{w_i}\right)^{\frac{1}{2}}v_j
    \right].
\end{align*}
where $v_i$ is the $i$th entry of the eigenvector $v$. With above relation, the partial derivative of the augmented Lagrangian \cref{eq: overall iterative algorithm for node weighted Laplacian 1} with respect to the $i$th node weight, ${w}_i$, is:
\begin{gather}
    \label{eq: partial derivative of the augmented Lagrangian}
    \begin{aligned}
        \frac{\partial {\mathscr{L}}_{{\rho}}}{\partial {w}_i}({\mathbf{w}},{\sigma})
         & =\frac{\partial \lambda_N\left(
            \hat{\mathcal{L}}_{{\mathbf{w}}}
            \right)}{\partial {w}_i}-{\rho}\left[
            {\max}\left\{
            0,1-\lambda_2\left(
            {\hat{\mathcal{L}}}_{{\mathbf{w}}}
            \right)+\frac{{\sigma}}{{\rho}}
            \right\}
            \right]\frac{\partial \lambda_2\left(
            \hat{\mathcal{L}}_{{\mathbf{w}}}
        \right)}{\partial {w}_i}                     \\
         & =\bar{v}_i \sum_{j\in\mathcal{N}_i}\left[
        \bar{v}_i-\left(\frac{w_j}{w_i}\right)^{\frac{1}{2}}\bar{v}_j
        \right]-\rho\left[
            \max\left\{
            0,1-\lambda_2\left(
            \hat{\mathcal{L}}_{{\mathbf{w}}}
            \right)+\frac{{\sigma}}{{\rho}}
            \right\}
            \right]\underline{v}_i \sum_{j\in\mathcal{N}_i}\left[
        \underline{v}_i-\left(\frac{w_j}{w_i}\right)^{\frac{1}{2}}\underline{v}_j
        \right].
    \end{aligned}
\end{gather}
where $\bar{v},\underline{v}$ are the unit eigenvectors corresponding to the largest and Fiedler eigenvalues, and $\bar{v}_i,\underline{v}_i$ mean the $i$th entry of the $\bar{v},\underline{v}$, respectively. The derivative of the augmented Lagrangian function with respect to the $i$th node weight ${w}_i$ requires the neighbors' weights $w_j$, the $i$th and $j$th entries of the unit largest and Fiedler eigenvectors, as well as the Fiedler eigenvalue. This property allows the distributed implementation of the weight updates.

For full distribution update, at each step of \cref{eq: distributed node weight update}, the $i$th agent should get the following information ready for calculation and to share with its neighbor:
\begin{inlineenum}
    \item The $i$th entry of the largest and Fiedler eigenvectors,
    \item The Fiedler value of $\mathcal{L}_\mathbf{w}$,
    \item Lagrangian multiplier $\sigma$ and parameter $\rho$, which are fixed during the iteration of \cref{eq: overall iterative algorithm for node weighted Laplacian 1} and shared among the network.
\end{inlineenum}

The first information is obtained inside the inner iteration, where the power iteration is employed to estimate the largest and Fiedler eigenvectors. Once the eigenvectors are estimated, the Fiedler value can be obtained simply by the definition of the eigenvalue. The parameter $\rho$ is a fixed constant that is set beforehand. The Lagrangian multiplier $\sigma$, on the other hand, relies on the uniform estimation of the Fiedler value and is updated by \cref{eq: overall iterative algorithm for node weighted Laplacian 2}. If each agent's copy of $\sigma$ starts with the same initial value, it will remain the same as long as every agent's estimation of the Fiedler value is the same, by \cref{eq: overall iterative algorithm for node weighted Laplacian 2}.
\subsection{Eigen-pair Estimation}
\label{subsec: Eigen-pair Estimation}
Time step $t$ is used to denote the steps in the update of \cref{eq: distributed node weight update}. Once $\mathbf{w}(t)$ is updated, new eigenvalues and eigenvectors for the new weighted Laplacian matrix $\mathcal{L}_{\mathbf{w}(t)}$ are required to be estimated. For this, the outer iteration $t$ is paused, and the inner iteration $T$ is started. Once the estimates are converged within the stopping criteria, the outer iteration $t$ is resumed using \cref{eq: distributed node weight update}.

The estimator is composed of three parts: the largest eigen-pair estimator, the Fiedler eigen-pair estimator, and the normalization block. Estimating the largest eigen-pair is leveraging the power iteration algorithm for the extreme eigenvalues, which can be performed in a distributed manner. The Fiedler eigen-pair estimator, on the other hand, is more challenging, for which we modify the ideas proposed by Bertrand et al \cite{Bertrand2013}. The normalization block is to distributedly find the eigenvectors with Euclidean norm equal to one.
\subsubsection{Max-consensus}

The power iteration requires vector normalization, particularly Euclidean, which is computationally intensive in distributed systems. To address this problem, we employ the max-consensus algorithm to obtain the infinity norm in finite time.

\begin{lemma}
    \label{lemma: max consensus} \cite{Lucchese2015,Nejad2009}
    Consider a graph $\mathcal{G}=(\mathcal{V},\mathcal{E})$, a vector $p\in\mathbb{R}^{|\mathcal{V}|}$ with the $i$th entry initialized on $i$th agent to be $p_i(0)$.  With the max-consensus algorithm \cref{eq: pure max consensus}:
    \begin{gather}
        \label{eq: pure max consensus}
        p_i(T+1)=\max_{j\in\mathcal{N}_i}\left\{
        p_j(T),p_i(T)
        \right\},\ \forall i\in\mathcal{V},
    \end{gather}
    the vector $p$ converges to the $\left\|p(0)\right\|_\infty\mathbf{1}_{|\mathcal{V}|}$ in at most $d(\mathcal{G})$ steps, i.e. $ p(T)=\left\|p(0)\right\|_\infty\mathbf{1}_{|\mathcal{V}|},\forall T\geq d(\mathcal{G})$.
\end{lemma}

As discussed in \cite{Garin2012}, the diameter of the graph can be estimated within $2d(\mathcal{G})$ iterations. With the knowledge of the graph diameter, the max-consensus algorithm can be terminated at most $d(\mathcal{G})$ steps. In this paper, we repeatedly employ the max-consensus algorithm to normalize the eigenvectors of the weighted Laplacian matrix.

\subsubsection{Largest Eigen-pair Estimator}
\label{subsec: Finding the largest eigenvalue}
A variation of the well-known power iteration method \cite{Corso1997,Lin2010} is used, due to the ease for distributed implementation. It only involves distributed elementary matrix-vector multiplication and scalar multiplication. Once converged, $\lambda_N$ can be obtained easily (\cref{subsec: Finding the eigenvalue}). Consider the following modified power iteration formula:
\begin{gather}
    \label{eq: central distributed power iteration using infty norm}
    \hat{x}(T+1)=\frac{\hat{\mathcal{L}}_\mathbf{w}x(T)}{\left\|\hat{\mathcal{L}}_\mathbf{w}x(T)\right\|_\infty}.
\end{gather}

Here, we rely on the max-consensus algorithm employed in \cref{eq: central distributed power iteration using infty norm} for normalization, which is a simple scaling of the 2-norm typically used. After each multiplication, the scaling will be done via the max-consensus algorithm in $d(\mathcal{G})$ steps.

First, a multiplication of the weighted Laplacian matrix $\hat{\mathcal{L}}_\mathbf{w}$ with the vector $x$ is performed, which is a distributed operation:
\begin{gather}
    \label{eq: distributed multiplication of the weighted Laplacian matrix}
    \left[
        \hat{\mathcal{L}}_\mathbf{w}x(T)
        \right]_i=w_i^{\frac{1}{2}}\sum_{j\in\mathcal{N}_i} \left[
    w_{i}^{\frac{1}{2}}x_i(T)-w_j^{\frac{1}{2}}x_j(T)
    \right],\ \forall i\in\mathcal{V},
\end{gather}
where the notation $[\cdot]_i$ means the $i$th entry of the vector. Then, the max-consensus algorithm \cref{eq: pure max consensus} is used to find the infinity norm of the vector $\hat{\mathcal{L}}_\mathbf{w}x(T+1)$, which takes $d(\mathcal{G})$ steps. Finally, the vector is normalized by dividing the infinity norm, which takes another step. Then $T$ becomes $T+1$ and the process is repeated. The iteration convergence ratio is known to be $\left({\lambda_{N-1}(\mathcal{L}_{\mathbf{w}})}/{\lambda_{N}(\mathcal{L}_{\mathbf{w}})}\right)$.

\subsubsection{Fiedler Eigen-pair Estimator}
The Fiedler eigen-pair is the second smallest eigenvalue of the matrix $\hat{\mathcal{L}}_\mathbf{w}$, and a method of ruling out the natural zero eigenvalue is needed. We leverage the idea proposed by Bertrand et al. \cite{Bertrand2013} and modify it for node weighted Laplacian matrix as well as the max consensus. Consider the auxiliary matrix
$$\mathcal{L}^c_\mathbf{w}=I-\frac{1}{\alpha}\hat{\mathcal{L}}_\mathbf{w},$$
where $\alpha>\frac{\lambda_2(\hat{\mathcal{L}}_\mathbf{w})+\lambda_N(\hat{\mathcal{L}}_\mathbf{w})}{2}$, such that $1-\frac{1}{\alpha}\lambda_2\left(
    \hat{\mathcal{L}}_\mathbf{w}
    \right)$ is the second extreme eigenvalue. A practical choice can be
\begin{gather}
    \alpha = \left\{
    \begin{array}{lc}
        \lambda_N(\mathcal{L}_\mathbf{w})                                & \text{if } \lambda_N(\mathcal{L}_\mathbf{w}) \text{ is found } \\
        \displaystyle\max_{i=\{1,\dots,N\}} 2\sum_{j\in\mathcal{N}_i}w_i & \text{otherwise}
    \end{array}
    \right.,
\end{gather}
The second value is found by applying the Gershgorin circle theorem \cite{Horn2012} for the weighted Laplacian matrix $\mathcal{L}_\mathbf{w}$, which can be distributed to every agent using max-consensus. With the appropriate choice of $\alpha$, consider the following dynamics
\begin{gather}
    \label{eq: central Fiedler eigen-pair estimator}
    y(T+1)=\left\{
    \begin{array}{lc}
        \frac{\hat{\mathcal{L}}_\mathbf{w} y(T)}{\left\|\hat{\mathcal{L}}_\mathbf{w} y(T)\right\|_\infty} & \qquad T \text{ mod } p=0, \\[4mm]
        \frac{\mathcal{L}^c_\mathbf{w} y(T)}{\left\|\mathcal{L}^c_\mathbf{w} y(T)\right\|_\infty}         & \qquad \text{otherwise},
    \end{array}
    \right.
\end{gather}
Recall the range space of $\hat{\mathcal{L}}_{\mathbf{w}}$ defined in \cref{eq: range space of the node weighted Laplacian matrix}. The main idea in \cref{eq: central Fiedler eigen-pair estimator} is to only search in the space $\mathcal{S}$, in which the Fiedler eigenvector of $\hat{\mathcal{L}}_{\mathbf{w}}$ corresponds to the extreme eigenvalue of the matrix $\mathcal{L}^c_\mathbf{w}$. Ideally, this can be done by multiplying $\hat{\mathcal{L}}_\mathbf{w}$ once, which corresponds to the case $p=\infty$, after which $y(T)\in\mathcal{S}$ for all following updates. However, in practice, the calculation error inevitably introduces the direction $\diag(\mathbf{w})^{-\frac{1}{2}}\mathbf{1}_N$ back, which will be amplified during the multiplication of $\mathcal{L}^c_\mathbf{w}$. Therefore, from time to time (every $p$ round), $\hat{\mathcal{L}}_{\mathbf{w}}$ is introduced to eliminate the $\diag(\mathbf{w})^{-\frac{1}{2}}\mathbf{1}_N$ in $y(T)$.

As demonstrated in \cite{Bertrand2013}, a rather tight bound can be obtained by choosing $p$ to be
\begin{gather}
    p\geq \frac{\alpha}{\lambda_2(\mathcal{L}_\mathbf{w})}=\frac{\lambda_N(\mathcal{L}_\mathbf{w})}{\lambda_2(\mathcal{L}_\mathbf{w})},
\end{gather}
which is exactly the loss function of the optimization problem here. In the very first iteration, when the agents know nothing about $\lambda_2(\mathcal{L})$, a lower bound \cite{Lu2007} of $\lambda_2(\mathcal{L})$ can be replaced to obtain $p$:
\begin{gather}
    \label{eq: choice of p with lower bound}
    \lambda_2(\mathcal{L}) \geq \frac{4d(\mathcal{G})}{d(\mathcal{G})+1}\Rightarrow p\geq\left\lceil \frac{\lambda_N(\mathcal{L})\left[d(\mathcal{G})+1\right]}{4d(\mathcal{G})}\right\rceil.
\end{gather}

The distributed implementation of \cref{eq: central Fiedler eigen-pair estimator} follows the same argument as the largest eigen-pair estimator in \cref{eq: distributed multiplication of the weighted Laplacian matrix}:
\begin{gather*}
    \left[
        \mathcal{L}^c_\mathbf{w}y(T)
        \right]_i= y_i(T)-\frac{1}{\alpha} w_i^{\frac{1}{2}}\sum_{j\in\mathcal{N}_i}\left[
    w_i^{\frac{1}{2}}y_i(T)-w_j^{\frac{1}{2}}y_j(T)
    \right],\ \forall i\in\mathcal{V}.
\end{gather*}
\subsubsection{Normalization}
When applying the gradient of the augmented Lagrangian method \cref{eq: partial derivative of the augmented Lagrangian}, it is necessary for the eigenvectors to satisfy the unit Euclidean-norm condition ($\|x\|_2 = 1$) as stated in \cref{lemma: derivative of the eigenvalue}, instead of the infinity norm used above. For this, an additional normalization step is then required. To apply the $l_2$ normalization, a consensus on the square of the states is performed using the original Laplacian matrix $\mathcal{L}$ as the communication matrix:
\begin{subequations}
    \label{Normalization procedure}
    \begin{align}
        \label{Normalization procedure for x}
        \tilde{x}_i(T+1) & = \tilde{x}_i(T) - \frac{1}{\beta} \sum_{j \in \mathcal{N}_i}  \left[ \tilde{x}_i(T) - \tilde{x}_j(T) \right], \\
        \label{Normalization procedure for y}
        \tilde{y}_i(T+1) & = \tilde{y}_i(T) - \frac{1}{\beta} \sum_{j \in \mathcal{N}_i}  \left[ \tilde{y}_i(T) - \tilde{y}_j(T) \right],
    \end{align}
\end{subequations}
where $x^*,y^*$ are the limit vectors obtained through \cref{eq: central distributed power iteration using infty norm,eq: central Fiedler eigen-pair estimator}, and initial conditions are $\tilde{x}_i(0) = (x_i^*)^2$ and $\tilde{y}_i(0) = (y_i^*)^2, \forall i\in\mathcal{V}$. As stated in the motivating example \cref{eq: optimal stepsize for real average consensus}, the optimal choice of $\beta$ is:
\begin{gather*}
    \beta=\frac{\lambda_N(\mathcal{L})+\lambda_2(\mathcal{L})}{2},
\end{gather*}
which should be distributedly available to all agents after each estimation of the largest and Fiedler eigen-pairs.
\subsection{Simulation Results}
\subsubsection{Node Weight Optimization}
\label{subsubsec: Node weight optimization}
In this section, we consider the graph topology shown in \cref{fig:graph topology}. The algorithms proposed in \cref{sec: Node weights} are implemented to optimize the node weights. The unweighted Laplacian matrix $\mathcal{L}$ exhibits the following baseline finite condition number:
\begin{gather}
    \label{eq: baseline finite condition number}
    \frac{\lambda_N}{\lambda_2}(\mathcal{L})=4.2689.
\end{gather}
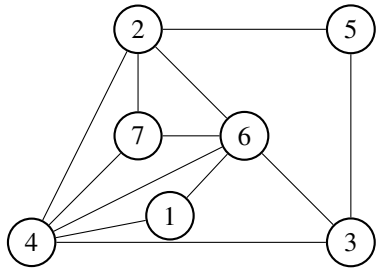
\begin{figure}[htbp]
    \centering
    \begin{tikzpicture}[every node/.style={circle,thick,draw},scale=0.7]
        % Nodes
        \node (1) at (0.6,0.5) {1};
        \node (2) at (0,4) {2};
        \node (3) at (4,0) {3};
        \node (4) at (-2,0) {4};
        \node (5) at (4,4) {5};
        \node (6) at (2,2) {6};
        \node (7) at (0,2) {7};
        % Edges based on adjacency matrix
        \draw (1) -- (4) -- (7) -- (2) -- (5) -- (3);
        \draw (1) -- (6) -- (2);
        \draw (4) -- (3) -- (6) -- (7);
        \draw (2) -- (4) -- (6);
    \end{tikzpicture}
    \caption{Graph topology used for simulation}
    \label{fig:graph topology}
\end{figure}

The simulation parameters used are listed in \cref{tab: parameter settings}. The evolution of the node weights over the outer iteration step $t$ is illustrated in \cref{fig: Node weight optimization}. The convergence behavior of the largest and Fiedler eigenvalues, as well as the resulting finite condition number, are depicted in \cref{fig: Node weight Eigenvalues}. The algorithm reaches the following optimized node-weighted Laplacian matrix and its corresponding finite condition number:
\begin{gather}
    \label{eq: node weight optimization result}
    \mathcal{L}_{\mathbf{w}^*}=\begin{bsmallmatrix}
        1.8539&         0&         0&   -0.9269  &       0 &  -0.9269 &        0\\
        0&    1.1287&         0&   -0.2822  & -0.2822 &  -0.2822 &  -0.2822\\
        0&         0&    1.2581&   -0.4194  & -0.4194 &  -0.4194 &        0\\
        -0.2442&   -0.2442&   -0.2442&    1.2208  &       0 &  -0.2442 &  -0.2442\\
        0&   -0.9346&   -0.9346&         0  &  1.8692 &        0 &        0\\
        -0.2442&   -0.2442&   -0.2442&   -0.2442  &       0 &   1.2208 &  -0.2442\\
        0&   -0.6192&         0&   -0.6192  &       0 &  -0.6192 &   1.8576\\
    \end{bsmallmatrix},\frac{\lambda_N}{\lambda_2}\left(
    \mathcal{L}_{\mathbf{w}^*}
    \right)=2.5584.
\end{gather}
Compared to the unweighted Laplacian matrix \cref{eq: baseline finite condition number}, the optimization of edge weights has reduced the finite condition number by $40\%$. On the other hand, if one uses an SDP solver to solve the centralized optimization problem \cref{eq: LMI formulation for node weight optimization}, the optimal finite condition number is ${\kappa}=2.5445$, which corresponds to a relative error of $0.55\%$.
\begin{figure}[htbp]
    \centering
    \begin{subfigure}{.49\linewidth}
        \centering
        \def\ymax{1.2}
        \def\ymin{0}
        \begin{tikzpicture}[scale=0.7]
            \begin{axis}[
                    xlabel={Outer iteration steps $t$},
                    ylabel={Node weights ${\mathbf{w}}$},
                    ymin=\ymin, % Ensure y-axis starts at 0
                    enlarge x limits={0},
                    ymax=\ymax,
                    cycle list={%
                            {blue, semithick},%
                            {red, semithick},%
                            {green, semithick},%
                            {orange, semithick},%
                            {purple, semithick},%
                            {brown, semithick},%
                            {pink, semithick},%
                            {gray, semithick},%
                            {cyan, semithick},%
                            {magenta, semithick},%
                            {lime, semithick},%
                            {olive, semithick},%
                            {teal, semithick},%
                            {violet, semithick},%
                            {yellow, semithick},%
                            {black, semithick},%
                            {darkgray, semithick},%
                            {lightgray, semithick}%
                        },
                    axis lines=box,
                    xtick pos =left,
                    ytick pos =bottom,
                    legend style={
                            at={(0.5,1.2)},  % (x,y) coordinates, 0.5 is center, 1.03 is slightly above the top
                            anchor=south,     % Anchor point of the legend
                            legend columns=7 % Optional: places legend entries in one row
                        },
                    clip=false,
                ]
                % Plot first y-axis values
                \pgfplotstableread[col sep=comma]{pics/nodeWeight.csv}{\nodeW}
                % Loop through columns 3 to end
                \pgfplotstablegetcolsof{\nodeW}
                \pgfmathtruncatemacro{\numcols}{\pgfplotsretval-1}

                \foreach \i in {1,...,\numcols} {
                        \addplot table[x index=0, y index=\i, col sep=comma] {\nodeW};
                        \addlegendentryexpanded{${w}_{\i}$}
                    }
                % Read the CSV file
                \pgfplotstableread[col sep=comma]{pics/nodeWeightMu.csv}{\nodeWMu}
                % Get number of rows
                \pgfplotstablegetrowsof{\nodeWMu}
                \pgfmathtruncatemacro{\lastrow}{\pgfmathresult-1} % Ensure it's a proper integer
                % Loop through rows
                \foreach \i in {0,...,\lastrow} {
                        % Get the x and y values from the CSV file
                        \pgfplotstablegetelem{\i}{0}\of\nodeWMu
                        \pgfmathsetmacro{\y}{\pgfplotsretval}  % Set \y as a number

                        \pgfplotstablegetelem{\i}{1}\of\nodeWMu
                        \pgfmathsetmacro{\x}{\pgfplotsretval}  % Set \x as a number

                        % Ensure the values of \x and \y are numbers before plotting
                        \addplot[black, thick, dashed] coordinates {(\x,\ymin) (\x,\ymax)};

                        % \pgfmathtruncatemacro{\ycoord}{\ymax+0.5 * mod(\i,2)+0.5 }
                        \pgfmathparse{(\i>=2)*(-0.4\mywidth+.3\mywidth*(\i-1))}
                        \edef\xoffset{\pgfmathresult}
                        \edef\doplot{\noexpand\node [black,above, xshift=\xoffset,align=center] at (\x,\ymax) {$\expandonce{\sigma}^{\i}=$\\$\pgfmathprintnumber[fixed,precision=3]{\y}$};}
                        \doplot
                    }
            \end{axis}
        \end{tikzpicture}
        \caption{Node Weight Update}
        \label{fig: Node weight optimization}
    \end{subfigure}
    \begin{subfigure}{.49\linewidth}
        \centering
        \def\ymax{7}
        \begin{tikzpicture}[scale=0.7]
            \begin{axis}[
                    xlabel={Outer iteration steps  $t$},
                    ylabel={eigenvalues \& finitie condition number},
                    ymin=0, % Ensure y-axis starts at 0
                    ymax=\ymax,
                    enlarge x limits={0},
                    cycle list name=color list,
                    axis lines=box,
                    xtick pos =left,
                    ytick pos =bottom,
                    legend style={
                            at={(0.5,1.2)},  % (x,y) coordinates, 0.5 is center, 1.03 is slightly above the top
                            anchor=south,     % Anchor point of the legend
                            legend columns=7 % Optional: places legend entries in one row
                        },
                    clip=false,
                ]
                % Plot first y-axis values
                \addplot table[x index=0, y index=1, col sep=comma] {pics/nodeWeightEigen.csv};
                \addlegendentry{$\lambda_N(\mathcal{L}_\mathbf{w})$}

                % Plot second y-axis values
                \addplot table[x index=0, y index=2, col sep=comma] {pics/nodeWeightEigen.csv};
                \addlegendentry{$\lambda_2(\mathcal{L}_\mathbf{w})$}

                % Plot the division of third y-axis by second y-axis
                \addplot table[x index=0, y expr=\thisrowno{2}/\thisrowno{1}, col sep=comma] {pics/nodeWeightEigen.csv};
                \addlegendentry{${\lambda_N(\mathcal{L}_\mathbf{w})}/{\lambda_2(\mathcal{L}_\mathbf{w})}$}

                % Read the CSV file
                \pgfplotstableread[col sep=comma]{pics/nodeWeightMu.csv}{\nodeWMu}
                % Get number of rows
                \pgfplotstablegetrowsof{\nodeWMu}
                \pgfmathtruncatemacro{\lastrow}{\pgfmathresult-1} % Ensure it's a proper integer
                % Loop through rows
                \foreach \i in {0,...,\lastrow} {
                        % Get the x and y values from the CSV file
                        \pgfplotstablegetelem{\i}{0}\of\nodeWMu
                        \pgfmathsetmacro{\y}{\pgfplotsretval}  % Set \y as a number

                        \pgfplotstablegetelem{\i}{1}\of\nodeWMu
                        \pgfmathsetmacro{\x}{\pgfplotsretval}  % Set \x as a number

                        % Ensure the values of \x and \y are numbers before plotting
                        \addplot[black, thick, dashed] coordinates {(\x,0) (\x,\ymax)};

                        \pgfmathparse{(\i>=2)*(-0.4\mywidth+.3\mywidth*(\i-1))}
                        \edef\xoffset{\pgfmathresult}
                        \edef\doplot{\noexpand\node [black,above, xshift=\xoffset,align=center] at (\x,\ymax) {$\expandonce{\sigma}^{\i}=$\\$\pgfmathprintnumber[fixed,precision=3]{\y}$};}
                        \doplot
                    }
            \end{axis}
        \end{tikzpicture}
        \caption{Eigenvalues and finite condition number of the node weight optimization}
        \label{fig: Node weight Eigenvalues}
    \end{subfigure}
    \caption{}
\end{figure}
\vspace{-5mm}
\subsubsection{MAS Consensus on different node weights}
In this section, we compare the consensus behavior of MAS under both unweighted and optimized node-weighted Laplacian matrices. Adopting the dynamics given in \cite{Col2018} and discretizing using a zero-order hold with a sampling time of $T = 0.5$:
\begin{gather*}
    \begin{bmatrix}
        \begin{array}{c|c|c}
            A   & B_1 & B_2 \\
            \hline
            C_1 & 0   & 0   \\
            \hline
            C_2 & D   & 0   \\
        \end{array}
    \end{bmatrix}=
    \begin{bmatrix}
        \begin{array}{cc|c|c}
            0.9232  & 0.4460 & 0.1125 & 0.4893  \\
            -0.4460 & 0.9232 & 0.4893 & -0.1125 \\
            \hline
            1       & 1      & 0      & 0       \\
            \hline
            0       & 1      & 0      & 0       \\
        \end{array}
    \end{bmatrix}
\end{gather*}

For the unweighted Laplacian matrix $\mathcal{L}$, we solve the optimization problem \cref{CL H infty optimization} at $\lambda_i= \lambda_2(\mathcal{L})=1.4892$ and $\lambda_N(\mathcal{L})=6.3574$, yielding the control gain matrices in \cref{eq: Generalized modified DOF with network difference signal} and \cref{eq: Generalized modified DOF with unweighted Laplacian} with $\gamma=1.1547$. In contrast, for the optimized node-weighted Laplacian matrix $\mathcal{L}_{\mathbf{w}^*}$ obtained in \cref{eq: node weight optimization result}, we solve the LMI problem in \cref{CL H infty optimization} at $\lambda_i=\lambda_2(\mathcal{L}_{\mathbf{w}^*})=1$ and $\lambda_N(\mathcal{L}_{\mathbf{w}^*})=2.5584$, resulting in \cref{eq: Generalized modified DOF with weighted Laplacian} with $\gamma=0.6957$.
\begin{figure}
    \centering
    \begin{subfigure}{.49\linewidth}
        \begin{gather*}
            \begin{bmatrix}
                \begin{array}{c|c}
                    A_c & B_c \\
                    \hline
                    C_c & D_c
                \end{array}
            \end{bmatrix}=
            \begin{bmatrix}
                \begin{array}{cc|c}
                    0.8346  & 0.1906  & -0.1413 \\
                    -0.8331 & -0.1892 & 0.1413  \\
                    \hline
                    -0.7914 & -2.2738 & -0.3995
                \end{array}
            \end{bmatrix}
        \end{gather*}
        \caption{Control gain matrices for unweighted Laplacian matrix $\mathcal{L}$}
        \label{eq: Generalized modified DOF with unweighted Laplacian}
    \end{subfigure}
    \begin{subfigure}{.49\linewidth}
        \begin{gather*}
            \begin{bmatrix}
                \begin{array}{c|c}
                    A_c^* & B_c^* \\
                    \hline
                    C_c^* & D_c^*
                \end{array}
            \end{bmatrix}=
            \begin{bmatrix}
                \begin{array}{cc|c}
                    0.8341  & 0.1901  & -0.6734 \\
                    -0.8340 & -0.1901 & 0.6734  \\
                    \hline
                    -0.7931 & -2.2754 & -0.5415
                \end{array}
            \end{bmatrix}
        \end{gather*}
        \caption{Control gain matrices for optimized weighted Laplacian matrix $\mathcal{L}_{\mathbf{w}^*}$}
        \label{eq: Generalized modified DOF with weighted Laplacian}
    \end{subfigure}
    \caption{}
\end{figure}
{With the noise $\xi_i$ are white noise that decays exponentially at the speed of $e^{-0.1k}$ over time,}
the simulation results are shown in \cref{fig: unweighted_sim,fig: weighted_sim}. From the same random initial state condition and zero initial control state, the trajectory of the MAS under the unweighted Laplacian is shown in \cref{fig: unweighted_sim}, while the trajectory of the MAS under the optimized node-weighted Laplacian is shown in \cref{fig: weighted_sim}. The oscillatory trajectories in the figure originate from the presence of unstable poles in the open-loop system, as consensus control is focused on achieving agreement. By optimizing the finite condition number of the Laplacian matrix, the resulting optimal controller not only accelerates convergence but also reduces the $l_2$ gain with respect to noise, illustrated in \cref{fig: optimal weighted_sim_inout,fig: unweighted_sim_inout}.

\begin{figure}[htbp]
    \centering
    % Plot first y-axis values
    \pgfplotstableread[col sep=comma]{pics/unweighted_sim.csv}{\unweightedSimData}
    \pgfplotstableread[col sep=comma]{pics/weighted_sim.csv}{\weightedSimData}
    % Loop through columns 3 to end
    \pgfplotstablegetcolsof{\unweightedSimData}
    \pgfmathtruncatemacro{\numcols}{\pgfplotsretval-1}
    \edef\ymax{3}
    \begin{subfigure}{.45\linewidth}
        \centering
        \begin{tikzpicture}
            \begin{groupplot}[
                    group style={group size=1 by 2},
                    width=\linewidth,
                    height=0.5\linewidth,
                    cycle list={%
                            {blue, semithick},%
                            {red, semithick},%
                            {green, semithick},%
                            {orange, semithick},%
                            {purple, semithick},%
                            {brown, semithick},%
                            {pink, semithick},%
                            {gray, semithick},%
                            {cyan, semithick},%
                            {magenta, semithick},%
                            {lime, semithick},%
                            {olive, semithick},%
                            {teal, semithick},%
                            {violet, semithick},%
                            {yellow, semithick},%
                            {black, semithick},%
                            {darkgray, semithick},%
                            {lightgray, semithick}%
                        }
                ]
                \nextgroupplot[
                    xlabel={$k$},
                    ylabel={Agent state $x_{i1}$},
                    % ymin=\ymin, % Ensure y-axis starts at 0
                    xmin=0,
                    ymax=\ymax,
                    ymin=-\ymax,
                    enlarge x limits=false,
                    axis lines=box,
                    xtick pos =left,
                    ytick pos =bottom,
                    legend style={
                            at={(0.45,1.05)},  % (x,y) coordinates, 0.5 is center, 1.03 is slightly above the top
                            anchor=south,     % Anchor point of the legend
                            legend columns=4 % Optional: places legend entries in one row
                        },
                    clip=false,
                ]
                \foreach \i in {1,3,...,\numcols} {
                        \addplot table[x index=0, y index=\i, col sep=comma] {\unweightedSimData};
                        \pgfmathparse{(\i+1)/2}
                        \edef\agentIndex{\pgfmathresult}
                        \addlegendentryexpanded{Agent ${\pgfmathprintnumber[fixed,precision=0]{\agentIndex}}$}
                    }

                \nextgroupplot[
                    xlabel={$k$},
                    ylabel={Agent state $x_{i2}$},
                    % ymin=\ymin, % Ensure y-axis starts at 0
                    xmin=0,
                    ymax=\ymax,
                    ymin=-\ymax,
                    enlarge x limits=false,
                    axis lines=box,
                    xtick pos =left,
                    ytick pos =bottom,
                    clip=false,
                ]
                \foreach \i in {2,4,...,\numcols} {
                        \addplot table[x index=0, y index=\i, col sep=comma] {\unweightedSimData};
                        \pgfmathparse{(\i)/2}
                        \edef\agentIndex{\pgfmathresult}
                    }
            \end{groupplot}
        \end{tikzpicture}
        \caption{Consensus of MAS with unweighted Laplacian matrix $\mathcal{L}$}
        \label{fig: unweighted_sim}
    \end{subfigure}
    \begin{subfigure}{.45\linewidth}
        \centering
        \begin{tikzpicture}
            \begin{groupplot}[
                    group style={group size=1 by 2},
                    width=\linewidth,
                    height=0.5\linewidth,
                    cycle list={%
                            {blue, semithick},%
                            {red, semithick},%
                            {green, semithick},%
                            {orange, semithick},%
                            {purple, semithick},%
                            {brown, semithick},%
                            {pink, semithick},%
                            {gray, semithick},%
                            {cyan, semithick},%
                            {magenta, semithick},%
                            {lime, semithick},%
                            {olive, semithick},%
                            {teal, semithick},%
                            {violet, semithick},%
                            {yellow, semithick},%
                            {black, semithick},%
                            {darkgray, semithick},%
                            {lightgray, semithick}%
                        }
                ]

                \nextgroupplot[
                    xlabel={$k$},
                    ylabel={Agent state $x_{i1}$},
                    % ymin=\ymin, % Ensure y-axis starts at 0
                    xmin=0,
                    ymax=\ymax,
                    ymin=-\ymax,
                    enlarge x limits=false,
                    axis lines=box,
                    xtick pos =left,
                    ytick pos =bottom,
                    legend style={
                            at={(0.45,1.05)},  % (x,y) coordinates, 0.5 is center, 1.03 is slightly above the top
                            anchor=south,     % Anchor point of the legend
                            legend columns=4 % Optional: places legend entries in one row
                        },
                    clip=false
                ]
                \foreach \i in {1,3,...,\numcols} {
                        \addplot table[x index=0, y index=\i, col sep=comma] {\weightedSimData};
                        \pgfmathparse{(\i+1)/2}
                        \edef\agentIndex{\pgfmathresult}
                        \addlegendentryexpanded{Agent ${\pgfmathprintnumber[fixed,precision=0]{\agentIndex}}$}
                    }

                \nextgroupplot[
                    xlabel={$k$},
                    ylabel={Agent state $x_{i2}$},
                    % ymin=\ymin, % Ensure y-axis starts at 0
                    xmin=0,
                    ymax=\ymax,
                    ymin=-\ymax,
                    enlarge x limits=false,
                    axis lines=box,
                    xtick pos =left,
                    ytick pos =bottom,
                    clip=false,
                ]
                \foreach \i in {2,4,...,\numcols} {
                        \addplot table[x index=0, y index=\i, col sep=comma] {\weightedSimData};
                        \pgfmathparse{(\i)/2}
                        \edef\agentIndex{\pgfmathresult}
                    }
            \end{groupplot}
        \end{tikzpicture}
        \caption{Consensus of MAS with optimized Laplacian matrix $\mathcal{L}_{\mathbf{w}^*}$}
        \label{fig: weighted_sim}
    \end{subfigure}
    \caption{}
\end{figure}

\begin{figure}[htbp]
    \centering
    \pgfplotstableread[col sep=comma]{pics/unweighted_inout.csv}{\unweightedInoutSimData}
    \pgfplotstablegetcolsof{\unweightedInoutSimData}
    \pgfplotstableread[col sep=comma]{pics/weighted_inout.csv}{\weightedInoutSimData}
    \pgfplotstablegetcolsof{\weightedInoutSimData}
    \pgfmathtruncatemacro{\numcols}{\pgfplotsretval-1}
    \pgfmathparse{(\numcols+1)/2}
    \edef\agentIndex{\pgfmathresult}
    \edef\legendYpos{1.1}
    \begin{subfigure}{.45\linewidth}
        \centering
        \begin{tikzpicture}[scale = 0.8]
            \begin{groupplot}[
                    group style={group size=1 by 2},
                    width=\linewidth,
                    height=0.5\linewidth,
                    cycle list={%
                            {blue, semithick},%
                            {red, semithick},%
                            {green, semithick},%
                            {orange, semithick},%
                            {purple, semithick},%
                            {brown, semithick},%
                            {pink, semithick},%
                            {gray, semithick},%
                            {cyan, semithick},%
                            {magenta, semithick},%
                            {lime, semithick},%
                            {olive, semithick},%
                            {teal, semithick},%
                            {violet, semithick},%
                            {yellow, semithick},%
                            {black, semithick},%
                            {darkgray, semithick},%
                            {lightgray, semithick}%
                        }
                ]
                \nextgroupplot[
                    xlabel={$k$},
                    ylabel={Disturbance $\xi_{i}$},
                    % ymin=\ymin, % Ensure y-axis starts at 0
                    xmin=0,
                    ymax=2.5,
                    ymin=-2.5,
                    enlarge x limits=false,
                    axis lines=box,
                    xtick pos =left,
                    ytick pos =bottom,
                    legend style={
                            at={(0.5,\legendYpos)},  % (x,y) coordinates, 0.5 is center, 1.03 is slightly above the top
                            anchor=south,     % Anchor point of the legend
                            legend columns=4 % Optional: places legend entries in one row
                        },
                    clip=false,
                ]

                \foreach \i in {1,2,...,\agentIndex} {
                        \addplot table[x index=0, y index=\i, col sep=comma] {\unweightedInoutSimData};
                        \pgfmathparse{(\i+1)/2}
                        \edef\agentIndex{\pgfmathresult}
                        \addlegendentryexpanded{Agent ${\pgfmathprintnumber[fixed,precision=0]{\agentIndex}}$}
                    }
                \nextgroupplot[
                    xlabel={$k$},
                    ylabel={General output $z_{i}$},
                    % ymin=\ymin, % Ensure y-axis starts at 0
                    xmin=0,
                    ymax=4,
                    ymin=-4,
                    enlarge x limits=false,
                    axis lines=box,
                    xtick pos =left,
                    ytick pos =bottom,
                    clip=false,
                ]
                \pgfmathparse{\agentIndex+1}
                \edef\agentIndexplusone{\pgfmathresult}
                \foreach \i in {\agentIndexplusone,...,\numcols} {
                        \addplot table[x index=0, y index=\i, col sep=comma] {\unweightedInoutSimData};
                    }
            \end{groupplot}
        \end{tikzpicture}
        \caption{Input and output of MAS with unweighted Laplacian matrix $\mathcal{L}$}
        \label{fig: unweighted_sim_inout}
    \end{subfigure}
    \begin{subfigure}{.45\linewidth}
        \centering
        \begin{tikzpicture}[scale =0.8]
            \begin{groupplot}[
                    group style={group size=1 by 2},
                    width=\linewidth,
                    height=0.5\linewidth,
                    cycle list={%
                            {blue, semithick},%
                            {red, semithick},%
                            {green, semithick},%
                            {orange, semithick},%
                            {purple, semithick},%
                            {brown, semithick},%
                            {pink, semithick},%
                            {gray, semithick},%
                            {cyan, semithick},%
                            {magenta, semithick},%
                            {lime, semithick},%
                            {olive, semithick},%
                            {teal, semithick},%
                            {violet, semithick},%
                            {yellow, semithick},%
                            {black, semithick},%
                            {darkgray, semithick},%
                            {lightgray, semithick}%
                        }
                ]
                \nextgroupplot[
                    xlabel={$k$},
                    ylabel={Disturbance $\xi_{i}$},
                    % ymin=\ymin, % Ensure y-axis starts at 0
                    xmin=0,
                    ymax=2.5,
                    ymin=-2.5,
                    enlarge x limits=false,
                    axis lines=box,
                    xtick pos =left,
                    ytick pos =bottom,
                    legend style={
                            at={(0.5,\legendYpos)},  % (x,y) coordinates, 0.5 is center, 1.03 is slightly above the top
                            anchor=south,     % Anchor point of the legend
                            legend columns=4 % Optional: places legend entries in one row
                        },
                    clip=false,
                ]

                \foreach \i in {1,2,...,\agentIndex} {
                        \addplot table[x index=0, y index=\i, col sep=comma] {\weightedInoutSimData};
                        \pgfmathparse{(\i+1)/2}
                        \edef\agentIndex{\pgfmathresult}
                        \addlegendentryexpanded{Agent ${\pgfmathprintnumber[fixed,precision=0]{\agentIndex}}$}
                    }
                \nextgroupplot[
                    xlabel={$k$},
                    ylabel={General output $z_{i}$},
                    % ymin=\ymin, % Ensure y-axis starts at 0
                    xmin=0,
                    ymax=4,
                    ymin=-4,
                    enlarge x limits=false,
                    axis lines=box,
                    xtick pos =left,
                    ytick pos =bottom,
                    clip=false,
                ]
                \pgfmathparse{\agentIndex+1}
                \edef\agentIndexplusone{\pgfmathresult}
                \foreach \i in {\agentIndexplusone,...,\numcols} {
                        \addplot table[x index=0, y index=\i, col sep=comma] {\weightedInoutSimData};
                    }
            \end{groupplot}
        \end{tikzpicture}
        \caption{Input and output of MAS with weighted Laplacian matrix $\mathcal{L}$}
        \label{fig: optimal weighted_sim_inout}
    \end{subfigure}
    \caption{}
\end{figure}
\section{Discussion and improvement}
\label{sec: Discussion and improvement}

\subsection{On edge weight optimization}
\label{sec: Edge weights}
The algorithm can also be applied to edge weight optimization, as well. Previous studies \cite{Gennaro2016,Shafi2012} focus on distributedly bounding the Fiedler eigenvalue of the edge weighted Laplacian matrix from below. Adopting a similar structure as in \cref{sec: Node weights}, problem \cref{eq: modified finite conditional number problem for node weighted Laplacian} can also be optimized by the distributed algorithm. The representation of the edge weighted Laplacian matrix is given by:
\begin{gather*}
    \mathcal{L}_\mathbf{w}=\mathcal{B}^T \diag(\mathbf{w})\mathcal{B}^T,
\end{gather*}
in which $\mathcal{B}$ is the incidence matrix of the graph \cite{Mesbahi2010}. With the above representation, the differential of a simple eigenvalue of the edge weight that connects the $i$th and $j$th nodes can be expressed as:
\begin{gather}
    \label{eq: partial derivative of the eigenvalue for edge weighted Laplacian}
    \frac{\partial}{\partial w_{ij}}\lambda(\mathcal{L}_\mathbf{w}) =v^T\left(\mathcal{B}^T \frac{\partial}{\partial w_{ij}} \diag(\mathbf{w})\mathcal{B}\right)v =\left(
    v_i-v_j
    \right)^2.
\end{gather}
where $v$ is the unit eigenvector corresponds to the simple eigenvalue $\lambda(\mathcal{L}_\mathbf{w})$. Since the differential of $w_{ij}$ only requires the local information of the $i$th and $j$th entries of the corresponding eigenvector, the gradient of the same Lagrangian, which is composed of the largest and Fiedler eigenvalues, can also be computed in a distributed manner:
\begin{gather}
    \label{eq: partial derivative of the augmented Lagrangian for edge weighted Laplacian}
    \begin{aligned}
        \frac{\partial}{\partial w_{ij}}\mathscr{L}_\rho(\mathbf{w}, \sigma^k)
         & =\frac{\partial}{\partial w_{ij}}\lambda_N(\mathcal{L}_\mathbf{w})+\frac{\partial}{\partial w_{ij}}\left\{
        \frac{\rho}{2}\left[\max \left\{0,1-\lambda_2(\mathcal{L}_\mathbf{w})+\frac{\sigma^k}{\rho}\right\}\right]^2
        \right\}                                                                                                      \\
         & = \left(
        \bar{v}_i-\bar{v}_j
        \right)^2-\rho\left[\max \left\{0,1-\lambda_2(\mathcal{L}_\mathbf{w})+\frac{\sigma^k}{\rho}\right\}\right]\left(
        \underline{v}_i-\underline{v}_j
        \right)^2.
    \end{aligned}
\end{gather}
where $\bar{v}$ and $\underline{v}$ are the unit eigenvectors corresponding to the largest and Fiedler eigenvalues, respectively. Equation \cref{eq: partial derivative of the augmented Lagrangian for edge weighted Laplacian} has a similar distributed structure as in the node weight case \cref{eq: partial derivative of the augmented Lagrangian}, which implies that the update law for the outer iteration remains fully distributed. The update law for the edge weight $w_{ij}$ is computed on both ends, i.e., the $i$th and $j$th agents. Each agent performs the same number of updates as its degree, which, for comparison, is always greater than or equal to the number of updates for the node weight optimization in \cref{sec: Node weights}. Similar distributed eigen-pair estimation in \cref{subsec: Eigen-pair Estimation} can be directly applied here, enabling a fully distributed implementation of the edge weight optimization algorithm.

\subsection{On non-simple eigenvalues}
\label{sec: non-simple eigenvalues}
In \cref{subsec: distributed node weight update}, only the simple largest and Fiedler eigenvalues of the Laplacian matrix are considered.
It is important to note that the proposed algorithm can be extended to handle non-simple eigenvalues as well. The biggest challenge is that at the point where the eigenvalues are not simple \cite{Clarke1990}, the differential of the largest (Fiedler) eigenvalue is not well defined, e.g., see the counterexamples in the paper \cite{Magnus1985}. To address this issue, the subgradient methods can be employed to obtain the subderivative of the largest (Fiedler) eigenvalue.

As proved in \cref{sec: Optimization Approaches.1a}, the function $\mathscr{L}_\rho(\mathbf{w}, \sigma^k)$ is convex in $\mathbf{w}$. Therefore, the subderivative of the optimization problem \cref{eq: overall iterative algorithm for node weighted Laplacian 1} can be expressed as (see Theorem 4.1.1 in paper \cite{Hiriart-Urruty1993}):
\begin{gather}
    \label{eq: subgradient of the optimization problem}
    \partial\lambda_N(C_w)={ \partial} \lambda_N\left(\mathcal{L}_{\mathbf{w}}\right)+{ \partial}\left\{
    \frac{\rho}{2}\left[\max \left\{0,1-\lambda_2\left(\mathcal{L}_{\mathbf{w}}\right)+\frac{\sigma^k}{\rho}\right\}\right]^2
    \right\}
\end{gather}

\begin{itemize}
    \item For $\lambda_N(\hat{\mathcal{L}}_\mathbf{w})$ (Example 2.8.7 in paper \cite{Clarke1990}):

          Note the definition of the largest eigenvalue:
          \begin{gather*}
              \lambda_N(\hat{\mathcal{L}}_\mathbf{w})=\sup_{\mathbf{x}\in\mathbb{R}^N,\|\mathbf{x}\|\leq 1}{f_x(\mathbf{w})=\mathbf{x}^T\hat{\mathcal{L}}_\mathbf{w}\mathbf{x}}.
          \end{gather*}

          The corresponding subgradient is, without any surprise:
          \begin{gather*}
              \partial \lambda_N(\hat{\mathcal{L}}_\mathbf{w})=\mathbf{Co} \cup \{
              \partial_{\mathbf{w}} f_x(\mathbf{w})\mid f_x(\mathbf{w})=\lambda_N(\hat{\mathcal{L}}_\mathbf{w}) \},\\
              \Rightarrow\partial \lambda_N(\hat{\mathcal{L}}_\mathbf{w})=\mathbf{Co}\left(
              \begin{bmatrix}
                  x^T\frac{\partial \hat{\mathcal{L}}_\mathbf{w}}{\partial w_1}x & \cdots & x^T\frac{\partial \hat{\mathcal{L}}_\mathbf{w}}{\partial w_N}x
              \end{bmatrix}\mid \hat{\mathcal{L}}_\mathbf{w}x=\lambda_N(\hat{\mathcal{L}}_\mathbf{w})x,\|x\|=1
              \right).
          \end{gather*}
          where $\mathbf{Co}$ means the convex hull of the set. In other words, any eigenvector $v\in\mathbb{R}^n$ corresponding to $\lambda_N$ can be considered when calculating the subgradient of the largest eigenvalue.
    \item For $ \frac{\rho}{2}\left[\max \left\{0,1-\lambda_2\left(\hat{\mathcal{L}}_{\mathbf{w}}\right)+\frac{\sigma^k}{\rho}\right\}\right]^2$:

          We benefit from the subgradient of the largest eigenvalue found above. The following equation holds all the time:
          \begin{gather*}
              1-\lambda_2\left(\hat{\mathcal{L}}_{\mathbf{w}}\right)+\frac{\sigma^k}{\rho}=\sup_{\begin{array}{c}
                      \mathbf{x}\in\mathbb{R}^N,\|\mathbf{x}\|\leq 1, \\
                      x^T\diag(\mathbf{w})^{\frac{1}{2}}\mathbf{1}_N=0
                  \end{array}}{h_x(\mathbf{w})=\left(
                  1-\frac{\sigma^k}{\rho}
                  \right)x^Tx+\mathbf{x}^T\hat{\mathcal{L}}_\mathbf{w}\mathbf{x}}.
          \end{gather*}
          Similarly, the subgradient of the function $1-\lambda_2\left(\hat{\mathcal{L}}_{\mathbf{w}}\right)+\frac{\sigma^k}{\rho}$ is:
          \begin{gather}
              \label{eq:subgradient for lambda2}
              \begin{aligned}
                  \partial \left[
                      1-\lambda_2\left(\hat{\mathcal{L}}_{\mathbf{w}}\right)+\frac{\sigma^k}{\rho}
                  \right] & =\mathbf{Co} \cup \{
                  \partial_{\mathbf{w}} h_x(\mathbf{w})\mid h_x(\mathbf{w})=1-\frac{\sigma^k}{\rho}-\lambda_2(\hat{\mathcal{L}}_\mathbf{w}) \}, \\
                          & =\mathbf{Co} \left\{
                  \begin{bmatrix}
                      x^T\frac{\partial \hat{\mathcal{L}}_\mathbf{w}}{\partial w_1}x & \cdots & x^T\frac{\partial \hat{\mathcal{L}}_\mathbf{w}}{\partial w_N}x
                  \end{bmatrix}\mid \hat{\mathcal{L}}_\mathbf{w}x=\lambda_2(\hat{\mathcal{L}}_\mathbf{w})x,\|x\|=1
                  \right\}
              \end{aligned}
          \end{gather}

          By point-wise maxima (Theorem 2.8.6 in \cite{Clarke1990}) and chain rule (Theorem 2.3.9 in \cite{Clarke1990}), the subgradient of the function $ \frac{\rho}{2}\left[\max \left\{0,1-\lambda_2\left(\hat{\mathcal{L}}_{\mathbf{w}}\right)+\frac{\sigma^k}{\rho}\right\}\right]^2$ is easily obtained as:
          \begin{gather*}
              \partial \left[
                  \frac{\rho}{2}\left[\max \left\{0,1-\lambda_2\left(\hat{\mathcal{L}}_{\mathbf{w}}\right)+\frac{\sigma^k}{\rho}\right\}\right]^2
                  \right]=\rho\max \left\{0,1-\lambda_2\left(\hat{\mathcal{L}}_{\mathbf{w}}\right)+\frac{\sigma^k}{\rho}\right\}
              \partial \left[
                  1-\lambda_2\left(\hat{\mathcal{L}}_{\mathbf{w}}\right)+\frac{\sigma^k}{\rho}
                  \right].
          \end{gather*}
          Similar to the largest eigenvalue, with the expression of \cref{eq:subgradient for lambda2}, any eigenvector $v\in\mathbb{R}^n$ corresponding to $\lambda_2$ can be used to calculate the subgradient of the function $ \partial \left[
                  1-\lambda_2\left(\hat{\mathcal{L}}_{\mathbf{w}}\right)+\frac{\sigma^k}{\rho}
                  \right].$
\end{itemize}

\subsection{Restriction on the graph}

As discussed in the Introduction and \cref{subsubsec: Average Consensus in MAS}, ensuring the Laplacian matrix has real eigenvalues is essential for meaningful analysis. While this is typically achieved by assuming the graph is undirected (\cref{assumption: undirected connected graph}), the optimization framework can be extended to certain classes of directed graphs where the Laplacian still has real eigenvalues. For example, acyclic graphs yield upper triangular Laplacians with real spectra, and leader-follower graphs with undirected follower subgraphs also satisfy this property. In such cases, the finite condition number remains well-defined, and the proposed methodologies can be adapted accordingly. The Laplacian for a directed graph can be constructed as:
\begin{gather*}
    \mathcal{L}_\mathbf{w}^d=\sum_{\begin{array}{c}
            \ell\in\{1,\cdots,|\mathcal{E}|\} \\
            e_\ell=(v_i,v_j)
        \end{array}}\mathbbm{1}_i w_{\ell}\mathcal{B}^T(e_\ell),
\end{gather*}
where $\mathbbm{1}_i$ is the indicator vector with $1$ at the $i$th entry. For directed graphs, analogous conditions to connectedness as in \cref{assumption: undirected connected graph} can be extended to imposing the existence of a spanning tree.
\subsection{Robustness to the node and link changes}
In the proposed distributed optimization framework, the only global graph information each agent requires is the graph diameter $d(\mathcal{G})$, which can be distributed among agents using max-consensus in $2d(\mathcal{G})$ steps \cite{Garin2012}. This minimal requirement allows the algorithm to flexibly accommodate dynamic changes in the network topology, such as the failure, removal or addition of nodes and edges. When a node or link change occurs, the network may become partitioned into several connected components. The preservation of overall connectedness as required by \cref{assumption: undirected connected graph} is not strictly necessary after such structural changes. Instead, each resulting connected component, by virtue of its own connectivity, can independently continue the distributed optimization process.

Upon detecting a structural change, during any operation, agents would transition into a safe mode to obtain the new diameter of the current connected component. Once the new diameter is established, agents can safely resume the distributed optimization with the controller best suited for the new topology. The algorithm's reliance on only local communication and minimal global information ensures that it remains robust and adaptable to network disruptions and reconfiguration. This property is particularly valuable in large-scale MASs, where network topology may change unpredictably due to failures, mobility, or environmental factors.

\subsection{Improvement for inner iteration}
The Lagrangian update and outer iteration in \cref{subsubsec: Node weight optimization} adhere to the conventional centralized solver, and thus have limited scope for enhancement. However, the inner iteration presents significant potential for optimization as the current algorithm involves pausing the gradient descent, pending the convergence of the inner iteration. This approach is suboptimal, as the convergence rate of the inner iteration critically influences the overall efficiency of the algorithm. The following are some potential improvements to the inner iteration process:

% \subsection{Leader-following structure}
\subsubsection{Enhancements for Power Iteration Efficiency}
In the context of \cref{eq: central distributed power iteration using infty norm,eq: central Fiedler eigen-pair estimator}, a significant portion of the computational time is allocated to the convergence of the max-consensus algorithm. For instance, for the vector $x(T)$, within each $d({\mathcal{G}})+2$ cycle, $d({\mathcal{G}})$ steps are dedicated to achieving max-consensus, while only two steps are utilized for matrix multiplication and normalization. This allocation is suboptimal, as the primary function of max-consensus is merely to normalize the vector. We propose the algorithms \cref{eq: central distributed power iteration using infty norm,eq: central Fiedler eigen-pair estimator} to facilitate understanding. In practice, it is not necessary to halt the power iteration for max-consensus after each matrix multiplication.

The normalization step in power iteration serves the purpose of preventing the vector entries from diverging to infinity or diminishing to negligible values. Only when the vector tends to converge is the introduction of max-consensus after each power iteration necessary. Otherwise, it is sufficient to divide the resulting vector by a constant value $\delta$. This $\delta$ should capture some sense of the norm of the vector, but does not need to be precise to the exact value; otherwise, the time-consuming max-consensus will be performed after each multiplication. The following idea illustrates one way to improve the efficiency of the power iteration:
\begin{gather*}
    \begin{aligned}
        x(T+1)         & =\frac{1}{\delta(T)}\mathcal{L}_\mathbf{w}x(T), \\
        \hat{x}_i(T+1) & = \left\{
        \begin{array}{lc}
            \max_{j\in\mathcal{N}_i}\left\{
            \hat{x}_j(T),\hat{x}_i(T)
            \right\},              & T \text{ mod } d(\mathcal{G}+1)\neq 1 \\
            \left|{x}_i(T)\right|, & T \text{ mod } d(\mathcal{G}+1)=1,
        \end{array}
        \right. , \forall i\in\mathcal{V}                                \\
        \delta(T+1)    & =\left\{
        \begin{array}{lc}
            \delta(T), & T \text{ mod } d(\mathcal{G}+1)\neq 0 \\
            \hat{x}_i  & T \text{ mod } d(\mathcal{G}+1)=0.
        \end{array}
        \right.
    \end{aligned}
\end{gather*}

When $T \text{ mod } d(\mathcal{G}+1)=0$, the parallel computation of each entry $\hat{x}_i$ converges to the infinity norm of $x(T-d(\mathcal{G}))$, which distributes the $\delta$ to each agent. This $\delta$ captures the infinity norm of $x$ $d(\mathcal{G})$ steps back. While it is not precise, it is sufficient to prevent the vector from diverging to infinity or diminishing to negligible values. This improvement can significantly reduce the time spent on max-consensus, thereby accelerating the convergence of the power iteration. When the two consecutive $\delta$ are close enough, agents can switch back to the original algorithm \cref{eq: central distributed power iteration using infty norm} to ensure the precision of the eigenvector.

\subsubsection{Selection of Initial Values for Eigenvalue Estimation}
The convergence of the distributed power iteration algorithms in \cref{eq: central distributed power iteration using infty norm,eq: central Fiedler eigen-pair estimator} relies on the standard assumption that the initial vectors $x(0)$ and $y(0)$ are not orthogonal to the respective target eigenvectors $\bar{u}$ and $\underline{u}$. In practice, random initialization is typically sufficient, as the probability of exact orthogonality is zero. Nevertheless, in the context of iterative optimization with small gradient descent steps, the structure of the problem allows for a more effective initialization strategy. Specifically, by initializing the eigenvector estimation at each iteration with the eigenvector estimates obtained from the previous step, the algorithm exploits the continuity of the eigenstructure under small perturbations of the weights. This reuse of previous estimates can significantly accelerate convergence and improve the numerical stability of the distributed power iteration process.

\section{Conclusions}
\label{sec: conclusions}

In this paper, we have presented a distributed optimization framework for the finite condition number of node-weighted Laplacian matrices in MASs. The proposed algorithm is designed to optimize the node weights in a distributed manner, ensuring that the agents can achieve consensus while minimizing the finite condition number of the Laplacian matrix. The algorithm is based on the augmented Lagrangian method and employs a projected gradient descent approach to ensure the positivity of the node weights.
The convergence of the algorithm is guaranteed under the assumption that the underlying graph is undirected and connected. The proposed method has been validated through numerical simulations, demonstrating its effectiveness in achieving consensus and optimizing the finite condition number of the Laplacian matrix.
%\backmatter
\bmsection*{Author contributions}
Both authors contributed to the study conception and design. Simulation and
analysis were performed by Yicheng Xu.
% \bmsection*{Acknowledgments}
% This is acknowledgment text.

% % \bmsection*{Financial disclosure}

% % None reported.

% \bmsection*{Conflict of interest}

% The authors declare no potential conflict of interest.

% \bmsection*{Data availability statement}
% The data that support the findings of this study are available from the corresponding author upon reasonable request.

% \bibliographystyle{IEEEtran}
\bibliography{wileyNJD-AMA,Uci}

% \bmsection*{Supporting information}

\appendix
\bmsection{Optimization Approaches}%
\label{sec: Optimization Approaches}
\vspace*{12pt}
The problems discussed in this paper can be framed as convex optimization problems. Some constraints are imposed to ensure the uniqueness of the optimal solution. It is well-known that the convex optimization problem can be solved by various methods, such as the interior point method, the penalty and barrier method, the primal-dual method, etc \cite{Luenberger2008}. Because the optimal solution is expected to be on the boundary of the search space, the augmented Lagrangian method is introduced. To further guarantee the positivity of the weights, the projected gradient descent method is introduced to facilitate solving the subproblem in the augmented Lagrangian method.
\bmsubsection{Augmented Lagrangian Method}
\label{subsec: Augmented Lagrangian Method}
Consider the following optimization problem:
\begin{gather}
    \label{eq: general optimization problem}
    \begin{aligned}
        \min_{x}    & \quad f(x) \\
        \text{s.t.} & \left\{
        \begin{array}{l}
            h(x) = 0, \\
            g(x)\leq 0.
        \end{array}
        \right.
    \end{aligned}
\end{gather}
where $h(x)\in\mathbb{R}^m$ and $g(x)\in\mathbb{R}^p$ are the equality and inequality constraints, respectively. Define the augmented Lagrangian function as:
\begin{gather*}
    \mathscr{L}_{\rho}(x, \mu, \sigma) = f(x) + \sum_{i=1}^m\frac{\rho}{2}\left[
        h_i(x) + \frac{\mu_i}{\rho}
        \right]^2 + \sum_{j=1}^p\frac{\rho}{2}\left[
        \max\left\{
        0, g_j(x) + \frac{\sigma_j}{\rho}
        \right\}
        \right]^2,
\end{gather*}

The augmented Lagrangian method is an iterative method solving the following subproblems:
\begin{gather*}
    \begin{aligned}
        x^{k+1}        & = \argmin_{x} \mathscr{L}_{\rho}(x, \mu^k,\sigma^k) &               \\
        \mu^{k+1}_i    & = \mu^k_i + \rho h_i(x^{k+1}),                      & i=1,\cdots,m, \\
        \sigma^{k+1}_j & = {\max}\left\{
        \sigma^k_j  + \rho g_j(x^{k+1}),0
        \right\},      & j=1,\cdots,p.
    \end{aligned}
\end{gather*}

Paper \cite{Birgin2014} shows that the augmented Lagrangian method converges to the points that satisfy the Karush-Kuhn-Tucker conditions of the original optimization problem; that is, in convex optimization problems, the augmented Lagrangian method converges to the global optimal solution.

\subsection{Projected Gradient Descent}
\label{subsec: Projected Gradient Descent}
It is common that the search space of the subproblem in the augmented Lagrangian method is a convex set. In this paper, the convex set is the non-negative orthant. Consider the following optimization problem:
\begin{gather}
    \label{eq: optimization problem w/ convex set}
        \min_{x}  f(x), \quad  s.t.   x\in\mathcal{X},
\end{gather}
where $\mathcal{X}$ is a closed convex set. Define the projection operator as $
    \mathcal{P}_{\mathcal{X}}(x) = \argmin_{y\in\mathcal{X}}\|y-x\|_2.
$
The unique existence of the projection operator follows from the Hilbert projection theorem \cite{Dimitri2015}. The projection gradient descent method is:
\begin{gather*}
    x^{k+1}= \mathcal{P}_{\mathcal{X}}\left[
        x^k-\alpha \nabla f(x^k)
        \right],
\end{gather*}
where $\alpha>0$ is the step size. The descent property of the projection gradient descent method follows by letting $z=x^k-\alpha \nabla f(x^k),x^{k+1}=\mathcal{P}_{\mathcal{X}}\left(
    z
    \right),$ and $x=x^k$  in Projection Theorem \cite{Ciarlet2013,Dimitri2015}:
\begin{gather*}
    \label{eq: Bourbaki-Cheney-Goldstein inequality}
    \langle z-\mathcal{P}_{\mathcal{X}}(z),x-\mathcal{P}_{\mathcal{X}}(z)\rangle \leq 0,\forall x \in \mathcal{X},\\
    \left\langle x^k-\alpha \nabla f(x^k)-x^{k+1},x^k-x^{k+1}\right\rangle \leq 0,\\
    \alpha\left\langle \nabla f(x^k),x^{k+1}-x^k\right\rangle \leq -\left\|x^k-x^{k+1}\right\|^2\leq 0.
\end{gather*}
Therefore, the projection gradient descent method is guaranteed to converge to the optimal solution for the convex optimization problem.
\bmsection{On optimization problems}%
\bmsubsection{Proof that \cref{eq: optimal stepsize for real average consensus} is the solution to \cref{eq: min max problem for average consensus}}
\label{sec: Optimization Approaches.1a}
\begin{proof}
    With \cref{assumption: undirected connected graph}, the eigenvalues of Laplacian can be ordered $0< \lambda_2(\mathcal{L})\leq \cdots \leq \lambda_N(\mathcal{L})$. Problem \cref{eq: optimal stepsize for real average consensus} can be rewritten as the following problem using the slack variable $\eta$:
    \begin{gather*}
        \min_{r\in\mathbb{R}_+}  \quad \eta   \quad
        \text{s.t. }            \quad  \left\{
        \begin{array}{lc}
            -\eta\leq 1-\frac{1}{r}\lambda_i(\mathcal{L}), & \quad \forall i=\{2,\cdots, N\} \\
            1-\frac{1}{r}\lambda_i(\mathcal{L})\leq \eta,  & \quad \forall i=\{2,\cdots, N\}
        \end{array}
        \right.
    \end{gather*}

    Since $\eta$ is greater than any $1-\frac{1}{r}\lambda_i(\mathcal{L})$, it is equivalent to bound it from below using the maximum value of $1-\frac{1}{r}\lambda_i(\mathcal{L})$, which yields $\eta\geq  1-\frac{1}{r}\lambda_2(\mathcal{L})$. Similarly, the negative $\eta$ needs to be bounded by the minimum, which yields $-\eta\leq 1-\frac{1}{r}\lambda_N(\mathcal{L})$. Therefore, an equivalent problem is:
    \begin{gather*}
        \min_{r\in\mathbb{R}_+}  \quad \eta    \quad
        \text{s.t. }             \quad  \left\{
        \begin{array}{l}
            -\eta\leq 1-\frac{1}{r}\lambda_N(\mathcal{L}), \\
            1-\frac{1}{r}\lambda_2(\mathcal{L})\leq \eta,
        \end{array}
        \right.
    \end{gather*}

    The Lagrangian function associated with the above problem is:
    \begin{gather*}
        \mathcal{L}(r,\eta,\mu_1,\mu_2)=\eta+\mu_1\left(
        \frac{1}{r}\lambda_N(\mathcal{L})-1-\eta
        \right)+\mu_2\left(
        1-\frac{1}{r}\lambda_2(\mathcal{L})-\eta
        \right),
    \end{gather*}
    where $\mu_1,\mu_2$ are the Lagrange multipliers. The KKT conditions for the above problem are:
    \begin{subequations}
        \label{eq: KKT conditions for min max problem}
        \begin{align}
            \label{eq: KKT conditions for min max problem 1}
            \frac{\partial \mathcal{L}}{\partial \eta}  =1-\mu_1-\mu_2                                                                    & =0,     \\
            \label{eq: KKT conditions for min max problem 2}
            \frac{\partial \mathcal{L}}{\partial r}     =-\frac{\mu_1}{r^2}\lambda_N(\mathcal{L})+\frac{\mu_2}{r^2}\lambda_2(\mathcal{L}) & =0,     \\
            \label{eq: KKT conditions for min max problem 3}
            \mu_1\left(
            \frac{1}{r}\lambda_N(\mathcal{L})-1-\eta
            \right)                                                                                                                       & =0,     \\
            \label{eq: KKT conditions for min max problem 4}
            \mu_2\left(
            1-\frac{1}{r}\lambda_2(\mathcal{L})-\eta
            \right)                                                                                                                       & =0,     \\
            \label{eq: KKT conditions for min max problem 5}
            \mu_1,\mu_2                                                                                                                   & \geq 0.
        \end{align}
    \end{subequations}

    Solving \cref{eq: KKT conditions for min max problem 1,eq: KKT conditions for min max problem 2} finds that $\mu_1=\frac{\lambda_2(\mathcal{L})}{\lambda_2(\mathcal{L})+\lambda_N(\mathcal{L})},\mu_2=\frac{\lambda_N(\mathcal{L})}{\lambda_2(\mathcal{L})+\lambda_N(\mathcal{L})}$, with which \cref{eq: KKT conditions for min max problem 3,eq: KKT conditions for min max problem 4} together imply \cref{eq: optimal stepsize for real average consensus}
\end{proof}

\bmsubsection{The LMI formulation of the problem}
\label{subsec: LMI formulation}
It can be shown that not only are the problems \cref{eq: modified finite conditional number problem for node weighted Laplacian} convex, but they can also be framed into LMI problems, and solved by modern SDP methods. This section serves as an alternate way to demonstrate that the problem discussed in the paper is convex. Furthermore, the optimal solution found by the LMI formulation will be used as the benchmark to evaluate the performance of the distributed algorithms. For the node weight optimization problem, adopt the notation as in the proof of \cref{lemma: convexity of the optimization problem}. The problem \cref{eq: modified finite conditional number problem for node weighted Laplacian} can be rewritten as:
\begin{gather}
    \label{eq: LMI formulation for node weight optimization}
    \min_{{\mathbf{w}}\in\mathbf{R}^N_+}  \quad {\kappa}\quad
    \text{s.t.}                               \quad {\kappa} I_{N-1}\succeq C^T\diag\left(
    \hat{\mathbf{w}}
    \right)C\succeq I_{N-1}.
\end{gather}

On the other hand, if one were to investigate the traditional edge weight, one can refer to papers \cite{Xiao2004,Shafi2012}. For completeness, here we also provide the LMI formulation for the edge weight optimization problem.
\begin{gather}
    \label{eq: LMI formulation for edge weight optimization}
    \min_{{\mathbf{w}}\in\mathbf{R}^N_+} {\kappa}, \quad \text{s.t.}  \quad {\kappa} I_{N}\succeq \mathcal{B}^T\diag\left(
    \hat{\mathbf{w}}
    \right)\mathcal{B}\succeq I_{N}-\frac{1}{N}\mathbf{1}_N\mathbf{1}_N^T.
\end{gather}

\bmsection{Algorithm related}
\subsection{Simulation parameters}
\label{subsec: Simulation parameters}
The key parameters are given in \cref{tab: parameter settings}. The overall optimization is considered to reach an optimal point when the absolute relative error of the Lagrange multiplier is less than the tolerance, e.g., $\left|\sigma^{k+1}-\sigma^k\right|\leq \varepsilon_{\sigma}$. For the outer iteration, the stopping criterion for the weight update in\cref{eq: overall iterative algorithm for node weighted Laplacian 1} is determined by any of the following conditions:
\begin{inlineenum}
    \item the maximum number of descent steps $t\leq t_{\max}$ is reached;
    \item the relative change in the objective function is less than the tolerance $\left|\lambda_N\left[
            \hat{\mathcal{L}}_{\mathbf{w}(t+1)}
            \right]-\lambda_N\left[
            \hat{\mathcal{L}}_{\mathbf{w}(t)}
            \right]\right|\leq \varepsilon_{\lambda_N}$;
    \item the relative change in the infinity norm satisfies $        \left\|
        \mathbf{w}(t+1)-\mathbf{w}(t)
        \right\|_\infty\leq \varepsilon_{\mathbf{w}}$ for four consecutive iterations.
\end{inlineenum}

On the other hand, the inner iteration for the eigen-pair estimation is considered to converge when the relative change in the eigenvalues is less than the tolerance, e.g., $x(T+1)-x(T)\leq \varepsilon_{x}$. The normalization process is considered to converge when the relative change in the normalized entries is less than the tolerance, e.g., $\left|\tilde{x}_i(T+1)-\tilde{x}_i(T)\right|\leq \varepsilon_{\tilde{x}}$. In the following simulation, an extra 50 steps are added to the bound of the periodic number $p$ in \cref{eq: choice of p with lower bound} to ensure the fast convergence of the Fielder eigenvector.

\begin{table}[htbp]
    \centering
    \begin{tabularx}{.95\textwidth}{>{\centering\arraybackslash\hsize=1.6\hsize}X|>{\centering\arraybackslash\hsize=0.4\hsize}X|>{\centering\arraybackslash\hsize=1.6\hsize}X|>{\centering\arraybackslash\hsize=0.4\hsize}X}
        Parameter setting                                         & Values         & Tolerance setting                                                            & Values            \\
        \hline
        Initial node weights ${\mathbf{w}}(0)$                    & $\mathbf{1}_N$ & Tolerance in infinity norm of  weights $\varepsilon_{\mathbf{w}}$            & $5\times 10^{-2}$ \\
        Initial Lagrange multiplier $\sigma(0)$                   & $0$            & Tolerance in Lagrange multiplier $\varepsilon_{\sigma}$                      & $ 10^{-1}$        \\
        Step size $\gamma$                                        & $0.001$        & Maximum number of descent steps $t_{\max}$                                   & $750$             \\

        Penalty parameter $\rho$                                  & $20$           & Tolerance in eigenvalue estimation $\varepsilon_{{x}},\varepsilon_{{y}}$     & $5\times 10^{-6}$ \\
        Tolerance in objective function $\varepsilon_{\lambda_N}$ & $10^{-3}$      & Tolerance in normalization $\varepsilon_{\tilde{x}},\varepsilon_{\tilde{y}}$ & $10^{-4}$         \\
    \end{tabularx}
    \caption{Parameter settings for the optimization of edge and node weights}
    \label{tab: parameter settings}
\end{table}
\subsection{Finding the eigenvalue}
\label{subsec: Finding the eigenvalue}
Suppose the $i$th entry of the eigenvectors $x$  of matrix $\mathcal{L}_\mathbf{w}$, is known to the agent $i$. It is known that $\mathcal{L}_\mathbf{w}x=\lambda x$. Therefore, for nonzero entry, the agent $i$ can estimate the corresponding eigenvalue $\lambda$ by performing the following operation:
\begin{gather*}
    \lambda=
    \frac{1}{x_i} \sum\limits_{
        \begin{small}
            \begin{array}{c}
                \ell\in\mathcal{N}_i \\
                e_\ell=(v_i,v_j)
            \end{array}
        \end{small}
    }w_\ell \left(
    x_i-x_j
    \right).
\end{gather*}
Since $x\neq0$, there must be a nonzero entry $x_j$ in the graph that is able to estimate the eigenvalue, any agent with zero $x_i$ can obtain the eigenvalue $\lambda$ using the max-consensus.

\bmsection*{Author Biography}
\begin{biography}{\includegraphics[width=76pt,height=76pt]{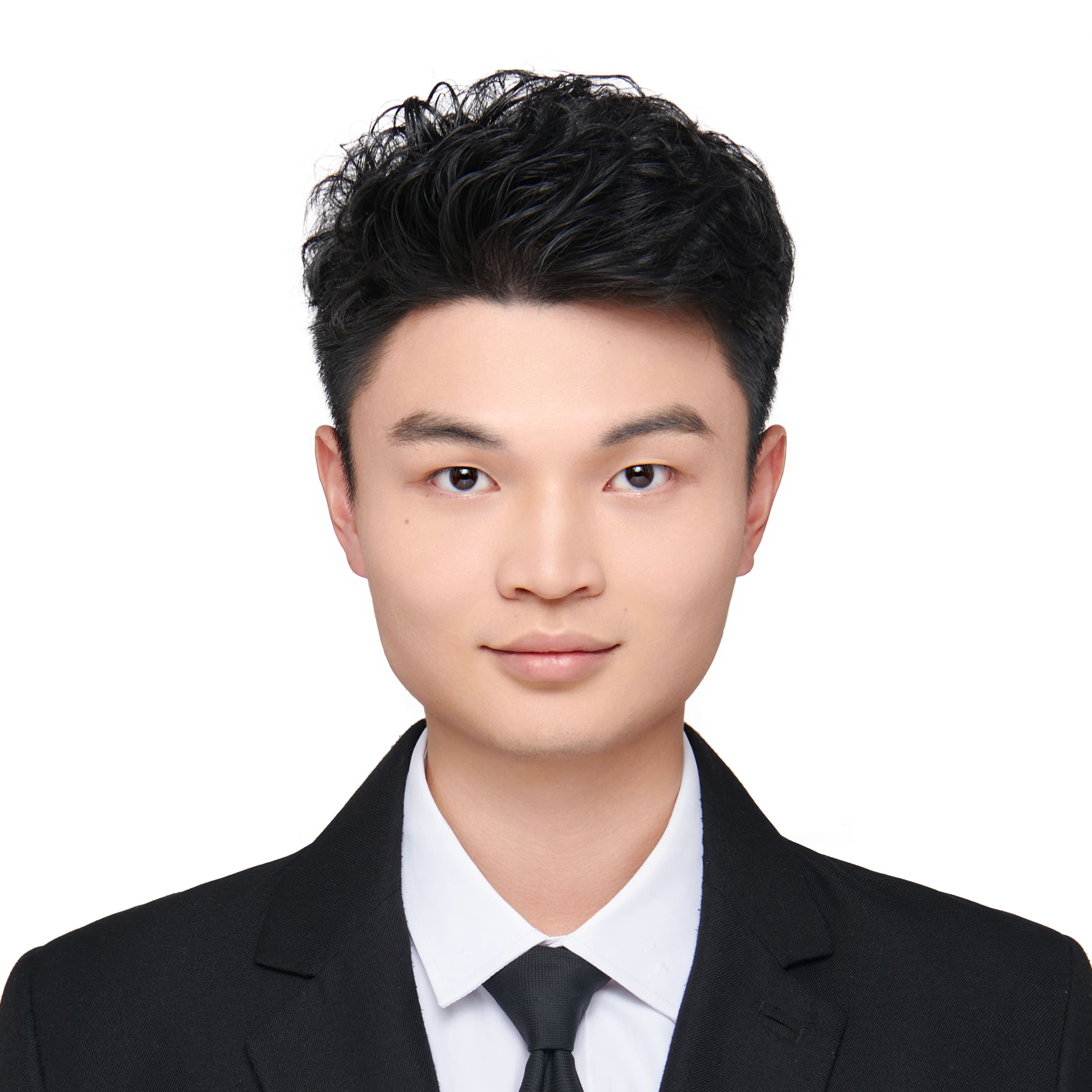}}{
        {\textbf{Yicheng Xu.} Yicheng Xu received the B.S. degree in automation from Southeast University, Nanjing, China in 2020 and the M.S. degree in mechanical engineering from the University of California, Irvine in 2021. He has been a Ph.D. student since 2021.   His research interests include distributed optimization, multi-agent systems, event-triggered control, and anti-windup control.}}
\end{biography}
\begin{biography}{\includegraphics[width=76pt,height=76pt]{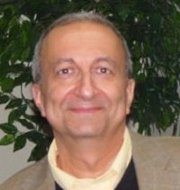}}{
        {\textbf{Faryar Jabbari.} Faryar Jabbari is on the faculty of the Mechanical and Aerospace Engineering Department of UCI. His research is in control theory and its
                applications. He has served as an associate editor for Automatica and
                IEEE Transactions on Automatic Control, the program chair for ACC-11
                and CDC-09, and the general chair for CDC-14.}}
\end{biography}
\end{document}